\documentclass[12pt]{amsart}

\usepackage{amsmath, amssymb, amsthm, amsfonts,verbatim}

\input xy
\xyoption{all}

\usepackage{hyperref}
\usepackage{graphicx}
\usepackage{color}

\swapnumbers
\numberwithin{equation}{section}

\theoremstyle{plain}
\newtheorem{theorem}[subsubsection]{Theorem}
\newtheorem{lemma}[subsubsection]{Lemma}
\newtheorem{prop}[subsubsection]{Proposition}
\newtheorem{cor}[subsubsection]{Corollary}
\newtheorem{conj}[subsubsection]{Conjecture}

\theoremstyle{definition}
\newtheorem{defn}[subsubsection]{Definition}
\newtheorem{cons}[subsubsection]{Construction}
\newtheorem{remark}[subsubsection]{Remark}
\newtheorem{remarks}[subsubsection]{Remarks}

\newtheorem{example}[subsubsection]{Example}

\def\AA{\mathbb{A}}
\def\BA{\mathbb{A}}

\def\CC{\mathbb{C}}

\def\GG{\mathbb{G}}
\def\G{\mathbb{G}}

\def\LL{\mathbb{L}}

\def\NN{\mathbb{N}}

\def\RR{\mathbb{R}}

\def\BR{\mathbb{R}}

\def\TT{\mathbb{T}}

\def\ZZ{\mathbb{Z}}
\def\BZ{\mathbb{Z}}

\def\calB{\mathcal{B}}
\def\calC{\mathcal{C}}

\def\calD{\mathcal{D}}
\def\D{\mathcal{D}}

\def\calF{\mathcal{F}}
\def\cF{\mathcal{F}}
\def\calG{\mathcal{G}}

\def\calL{\mathcal{L}}

\def\calO{\mathcal{O}}
\def\cO{\mathcal{O}}
\def\calP{\mathcal{P}}
\def\calQ{\mathcal{Q}}
\def\cQ{\mathcal{Q}}

\def\cS{\mathcal{S}}

\def\ZZ{\mathbb{Z}}

\newcommand\frH{\mathfrak{H}}

\newcommand\frP{\mathfrak{P}}

\newcommand\tilP{\widetilde{P}}

\newcommand\tilW{\widetilde{W}}

\newcommand\tilcalO{\widetilde{\mathcal{O}}}

\newcommand\tilcalP{\widetilde{\mathcal{P}}}
\newcommand\tilLL{\widetilde{\mathbb{L}}}

\newcommand{\Coh}{\textup{Coh}}
\newcommand{\coker}{\textup{coker}}

\newcommand\Fun{\textup{Fun}}

\newcommand\id{\textup{id}}

\newcommand\Mod{\textup{Mod}}

\newcommand{\Perf}{\textup{Perf}}

\newcommand\Spec{\textup{Spec}}

\newcommand\Hom{\textup{Hom}}

\newcommand\nc{\newcommand}
\nc\on{\operatorname}
\nc\ol{\overline}
\nc\ul{\underline}
\nc\olY{\ol{Y}}

\nc\oo{\infty}

\nc\Cone{\mathit{Cone}}
\nc\ssupp{\mathit{ss}}
\nc\risom{\stackrel{\sim}{\to}}
\nc\Sh{\mathit{Sh}}
\nc\un{\diamondsuit}
\nc\orient{\mathit{or}}
\nc\sing{\mathit{sing}}
\nc\MF{\on{MF}}
\nc\inthom{\mathit{Hom}}
\nc\colim{\on{colim}}
\nc\wmsh{\mu\mathit{sh}^{wr}}

\nc\conv{\mathit{conv}}
\nc\trad{\mathit{inf}}
\nc\wrap{\mathit{wr}}
\nc\Log{\on{Log}}
\nc\Arg{\on{Arg}}

\begin{document}


\title[Mirror symmetry for honeycombs]{Mirror symmetry for honeycombs}

\author{Benjamin Gammage and David Nadler}
\address{Department of Mathematics\\University of California, Berkeley\\Berkeley, CA  94720-3840}
\email{bgammage@math.berkeley.edu}
\email{nadler@math.berkeley.edu}

\begin{abstract} We prove a homological mirror symmetry equivalence between the $A$-brane category of the pair of pants,
computed as a wrapped microlocal sheaf category, and the $B$-brane category of its mirror LG model, understood as a category of matrix factorizations. 
The equivalence improves upon prior results in two ways:  it intertwines evident affine Weyl group symmetries on both sides,
and it exhibits the relation of wrapped microlocal sheaves along different types of Lagrangian skeleta  for the same hypersurface. 
The equivalence proceeds through the construction of a combinatorial realization of the $A$-model via arboreal singularities. 
The constructions here represent the start of a program to generalize to higher dimensions many of the structures which have appeared in topological approaches to Fukaya categories of surfaces.
\end{abstract}

\maketitle


\tableofcontents


\section{Introduction}

This paper fits into the framework of  Homological Mirror Symmetry, as introduced in \cite{kontsevich1995homological} and expanded in \cite{K98,HV00,KKOY}. The  formulation of interest to us 
relates the $A$-model of a hypersurface $X$ in a toric variety to the mirror Landau-Ginzburg $B$-model
of a toric variety $X^\vee$ equipped with superpotential $W^\vee\in \calO(X^\vee).$ 
Following Mikhalkin~\cite{M}, a distinguished ``atomic" case is when the hypersurface is the
pair of pants 
$$
\calP_{n-1}=\{z_1+\cdots+z_n+1=0\}\subset(\CC^*)^n\cong T^*(S^1)^n
$$
with mirror  Landau-Ginzburg model $(\AA^{n+1},z_1\cdots z_{n+1})$.
In this paper, we will also be interested in the universal abelian cover $\tilcalP_{n-1}$ of the pair of pants, which fits in the Cartesian diagram
$$
\xymatrix{
\tilcalP_{n-1}\ar[r]\ar[d]&T^*\RR^n\ar[d]\\
\calP_{n-1}\ar@{^{(}->}[r]&T^*(S^1)^n
}
$$
as the pullback of $\calP_{n-1}$ along the universal covering map $T^*\RR^n\to T^*(S^1)^n$; it has mirror a torus-equivariant version of the Landau-Ginzburg model $(\AA^{n+1},z_1\cdots z_{n+1}).$
%
%

This paper expands upon prior mirror symmetry equivalences for pairs of pants found in \cite{seidelgenustwo,AAEKO,Sher1};
however, it differs from those in its understanding of the $A$-model.
The traditional mathematical realization of the $A$-model is the Fukaya category, with objects decorated Lagrangian submanifolds, morphisms their decorated intersections, and structure constants defined by  integrals over moduli spaces of pseudoholomorphic polygons. There is increasing evidence (for example, \cite{SGoHA,BonOr,FLTZ2,NZ,Ncs, Tam,GPS1,St,GPS2,GPS3})
that the Fukaya category of a Weinstein manifold  is equivalent to microlocal sheaves (as developed by Kashiwara-Schapira \cite{KS}) along a Lagrangian skeleton.
In this paper, we follow \cite{Nwms} and study the $A$-model of the pair of pants  in its guise as wrapped microlocal sheaves.

A calculation of microlocal sheaves on a skeleton for the pair of pants was performed already in \cite{Nwms}; our calculation here involves a different skeleton, which is of independent interest. The skeleton we study here is more symmetrical, having an action of the symmetric group $\Sigma_{n+1}$ rather than just $\Sigma_n$; but more importantly, the skeleton here is of a different ``flavor'' to the one constructed there. The calculations from \cite{Nwms} are well-adapted to considerations of mirror symmetry which relate a hypersurface in $(\CC^\times)^n$ to a toric degeneration and were used in \cite{GS} for this purpose. The skeleton we study in this paper is more adapted to mirror symmetry equivalences which relate a hypersurface in $(\CC^\times)^n$ to a Landau-Ginzburg model. The first sort of skeleton can be considered as a ``degeneration'' of the second--indeed, the relation between these two flavors of skeleton is very interesting and will be studied further in future work.

The skeleton from this paper is very well-suited to a combinatorial perspective, since singularities are all arboreal in the sense of \cite{Narb}. The form of our calculations should be understood as a paradigm for extending the substantial literature devoted to understanding Fukaya categories of Riemann surfaces through topological skeleta and ribbon graphs  (for example, \cite{STZ,B,DK,HKK,PS}) to higher-dimensional examples.

Moreover, the type of skeleton described here has close relations to the dimer models which have appeared in earlier mirror symmetry contexts (\emph{e.g.}, \cite{FHKV,FU}); in future work, we hope to explore further the relation between skeleta and dimer models, along with generalizations to higher dimensions.
This correspondence was noticed (in a slightly different form) in \cite{FU}, and in Section~\ref{subsec:skel} we make use of the Lefschetz fibrations described in that paper.

\subsection{Symplectic geometry}

\subsubsection{Cotangent bundles} 

Fix a characteristic zero coefficient field $k,$ and let $L\subset T^*X$ be a closed conic Lagrangian submanifold of a cotangent bundle. There are conic sheaves of dg categories $\mu\Sh_L^\Diamond$ and $\mu\Sh_L$ on $T^*X$, localized along $L,$ which to a conic open set $\Omega\subset T^*X$ assign, respectively, the dg category $\mu\Sh_L^\Diamond(\Omega)$ of unbounded-rank microlocal sheaves and the dg category $\mu\Sh_L(\Omega)$ of traditional microlocal sheaves along $L\cap \Omega.$
\begin{defn}
The category $\mu\Sh_L^{wr}(\Omega)$ of \emph{wrapped microlocal sheaves} along $L\cap \Omega$ is the category $\mu\Sh_L^\Diamond(\Omega)^c$ of compact objects inside of $\mu\Sh_L^\Diamond(\Omega).$
\end{defn}

\begin{prop}[\cite{Nwms} Proposition 3.16]
The assignment $\Omega\mapsto\mu\Sh_L^{wr}(\Omega)$ forms a cosheaf $\mu\Sh_L^{wr}$ of dg categories on $T^*X,$ localized on $L.$
\end{prop}

If $x$ is a smooth point of $L$ and $\Omega$ is a contractible conic neighborhood of $x$, then the stalk of the sheaf $\mu\Sh_L^\Diamond$ is equivalent to the dg category $\Mod_k$ of (unbounded-rank) $k$-modules, while the stalk of $\mu\Sh_L$ and the costalk of $\mu\Sh_L^w$ at $x$ are both equivalent to the dg category $\Perf_k$ of perfect $k$-modules.

If $x$ is a singular point, the local calculation is more complicated, but this calculation has already been performed in \cite{Narb} for a certain class of Legendrian singularities termed arboreal. In this paper we will only be concerned with the $A_n$ arboreal singularity $\calL_{A_n},$ a certain singular Legendrian in the projectivized cotangent bundle $T^\infty(\RR^n)$ which is homeomorphic to $\Cone(\text{sk}_{n-2}\Delta^n),$ the cone on the $(n-2)$-skeleton of an $n$-simplex.
\begin{prop}[\cite{Narb}] Let $L$ be a conic Lagrangian in $T^*\RR^n$ which is locally equivalent, near a point $x\in L,$ to the cone on $\calL_{A_n}.$ Then to a neighborhood of $x$, the sheaf $\mu\Sh_L$ and the cosheaf $\mu\Sh_L^{wr}$ each assign the category $A_n\on{-Perf}_k$ of perfect modules over the $A_n$ quiver.
\end{prop}

\subsubsection{Weinstein manifolds} Let $W$ be a Weinstein manifold. The Weinstein structure of $W$ endows it with a Lagrangian skeleton $\Lambda$, onto which $W$ deformation retracts. 

Since our definition of wrapped microlocal sheaf categories applies only in the setting of cotangent bundles, in order to apply it here we have to relate the geometry of our Weinstein manifold to the geometry of a cotangent bundle. Let $\Lambda$ be the skeleton of a Weinstein manifold $W$ and $U$ an open neighborhood of $\Lambda$ which is conic for the flow of the Liouville vector field of $W$, and suppose that there exists a manifold $X$ and a closed conic Lagrangian $L\subset T^*X$ such that $U$ is exact symplectomorphic to a neighborhood $\Omega$ of $L$ by a symplectomorphism taking $\Lambda$ to $L.$
\begin{defn}In the situation described above, the category $\mu\Sh_L^w(\Omega)$ is the \emph{microlocal $A$-model category} associated to the Weinstein manifold $W$. (To make explicit the dependence on Weinstein structure, we will sometimes call this category the wrapped microlocal $A$-model category of $W$ associated to $\Lambda.$)
\end{defn}

In practice, the symplectomorphism relating $\Lambda$ to a conic Lagrangian in a cotangent bundle might not exist. However, such symplectomorphisms always exist locally, so that we can obtain a cosheaf of categories on the skeleton by defining these categories locally, gluing them together, and checking that the resulting cosheaf didn't depend on choices. A more detailed explanation of our expectations can be found in Conjecture \ref{conj-cosheaf}.


One skeleton for the pair of pants $\calP_{n-1}$ was described in \cite{Nwms}, where it was used to prove a mirror symmetry equivalence. In this paper, we study a more symmetric skeleton of the pair of pants, which we can describe using the geometry of the permutohedron.

Let $V_n$ be the quotient of $\RR^{n+1}$ by the span of the vector $\lambda_1+\cdots+\lambda_{n+1},$ where $\{\lambda_i\}$ is the standard coordinate basis of $\RR^{n+1}.$ 

\begin{defn}
The \emph{n-permutohedron} $\frP_n \subset V_n$ is the convex poytope obtained as the convex hull
$$
\xymatrix{
\frP_n = \conv\{ \sigma\cdot \left(\frac{1}{n+1}\sum_{a=1}^{n+1}a\lambda_a\right) \in V_n \;|\; \sigma\in \Sigma_{n+1}\}.
}
$$
\end{defn}

The $n$-permutohedron is an $n$-dimensional polytope, and it is a remarkable fact that the permutohedron actually tiles $V_n$. We denote by $\frH_{n-1}$ the union of all translates of the boundary $\partial \frP_n$ along this tiling and call this space the \emph{honeycomb}. (When $n=2$, the honeycomb $\frH_1$ is actually the boundary of the hexagon tiling of the plane.) Then the main result of section 4 of this paper is a stronger version of the following, which we obtain as a $\ZZ/(n+1)\ZZ$ quotient of a calculation performed in \cite{FU}:
\begin{prop}[Corollary \ref{cor:honeycombskel} below]
The cover $\tilcalP_{n-1}$ admits a skeleton $\tilLL_{n-1}$ whose image under the (cover of the) argument map $\widetilde{\Arg}:T^*\RR^n\to \RR^n$ is the honeycomb $\frH_{n-1}.$
\end{prop}

This result and the discussion above justifies our modeling of the wrapped Fukaya category of $\tilcalP_{n-1}$ as the global sections of a certain cosheaf $\calQ_{n-1}^{wr}$ of dg categories on $\frH_{n-1}$ (and the infinitesimally wrapped Fukaya category as the global sections of a certain sheaf $\calQ_{n-1}^{inf}$ of dg categories on $\frH_{n-1}$).

\subsection{Combinatorics}

As mentioned above, the sheaf and cosheaf $\calQ_{n-1}^{inf}$ and $\calQ_{n-1}^{wr}$ assign $\Perf_k$ to a smooth point of $\frH_{n-1},$ but to know their descriptions over the whole skeleton $\frH_{n-1},$ we need to understand its singularities. These turn out to be singularities we already understand:

\begin{prop}[Proposition \ref{prop:sings} below]A neighborhood of a point in a codimension-$m$ face of a permutohedron in $\frH_{n-1}$ is stratified homeomorphic to the product of $\RR^{n-m}$ with the $A_m$ arboreal singularity $\calL_{A_m}$.
\end{prop}

(In fact, to prove the equivalence of (co)sheaves of categories we need, the above homeomorphism isn't sufficient; it's necessary to check the stronger statement that a neighborhood of the corresponding point in the skeleton is contactomorphic, up to a smoothing, to the product $\RR^{n-m}\times\calL_{A_m},$ or that it admits the same category of microlocal sheaves. This is Proposition~\ref{prop:miccat} below.)

The upshot is that all the singularities of the skeleton $\tilLL_{n-1}$ of $\tilcalP_{n-1}$ are of type $A_m$ (for various $m$) which proves to be extremely convenient for calculation of the Fukaya category. As the name suggests, the sections of $\calQ^{inf}_{n-1}$ or $\calQ^{wr}_{n-1}$ on a neighborhood of the $A_m$ singularity $\calL_{A_m}$ are equivalent to the dg category $A_m\on{-Perf}_k$ of perfect modules over the $A_m$ quiver. 

\begin{remark}The equivalence of the category associated to an $A_m$ singularity with the category $A_m\on{-Perf}_k$ is noncanonical, reflecting the fact that the category $A_m\on{-Perf}_k$ has a $\ZZ/(m+1)$ symmetry. Moreover, due to the standard appearance of the ``metaplectic anomaly" in the construction of Fukaya categories, we cannot keep global track of the integer grading on this category without making additional choices, so in practice we will only ever work with a $\ZZ/2$-graded version of this local category, which is (noncanonically) equivalent to the $\ZZ/2$-dg category $(A_m\on{-Perf}_k)_{\BZ/2}.$ 
\end{remark}

We can summarize the above discussion as describing the following procedure: Stratify the space $\frH_{n-1}$ by singularity type, and let $P(\frH_{n-1})$ be the poset corresponding to the stratification. The Fukaya category associated to the skeleton $\frH_{n-1}$ is the global sections of a sheaf/cosheaf, taking values in the $\ZZ/2$-dg category $\BZ/2\on{-dgst}_k$ of $\ZZ/2$-dg categories, which assigns to a neighborhood of a point in a codimension-$(m-1)$ stratum of $\frH_{n-1}$ a category equivalent to $(A_m\on{-Perf}_k)_{\BZ/2}.$ 
\begin{defn} The \emph{wrapped} and \emph{infinitesimally wrapped combinatorial Fukaya categories} associated to the pair of pants are the categories
$$
\xymatrix{
Q_{n-1}^{wr}=\on{Idem}(\colim(P(\frH_{n-1})^{op} \ar[r]^-{\calQ^{wr}_{n-1}} & \BZ/2\on{-dgst}_k)) &
Q_{n-1}^{inf}=\lim(P(\frH_{n-1}) \ar[r]^-{\calQ^{inf}_{n-1}}&\BZ/2\on{-dgst}_k)
}
$$
defined as (idempotent-completed) global sections of the cosheaf $\calQ^{wr}_{n-1}$ and sheaf $\calQ^{inf}_{n-1},$ respectively, over the honeycomb $\frH_{n-1}.$
\end{defn}

Objects in the infinitesimally wrapped category, which is defined as global sections of the sheaf $\calQ_{n-1}^{inf},$ have a clearer geometric meaning: heuristically, an object of this category can be described as the data of an object of $(\Perf_k)_{\BZ/2}$ at each facet in $\frH_{n-1},$ exact triangles among these at codimension-2 faces of permutohedra, and higher compatibilities given by codimension $k$ faces. (For instance, the compatibility at a codimension 3 face involves assembling the four exact triangles around the face into an ``octahedral axiom diagram.")

For $F$ a facet in $\frH_{n-1}$ and $\xi$ a choice of codirection at $F$ (breaking the $\ZZ/2\ZZ$ indeterminacy in the category associated to $F$), we have a \emph{stalk functor} $\phi_{F,\xi}:Q^{inf}_{n-1}\to(\Perf_k)_{\BZ/2}$ taking an object of $Q^{inf}_{n-1}$ to the object of $(\Perf_k)_{\BZ/2}$ which is placed at the facet $F.$ If we understand the data of an object in $Q_{n-1}^{inf}$ as recording a path of a Lagrangian running along the honeycomb $\frH_{n-1}$, then the stalk of an object along a facet $F$ records how many times the Lagrangian runs along $F$.

The following class of objects in $Q_{n-1}^{inf}$ is easy to describe and are very useful in proving mirror symmetry for this category:

\begin{defn}Let $B$ be a connected subset of $\frH_{n-1}.$ A \emph{rank-one brane along $B$} is an object $\calF$ in $Q^{inf}_{n-1}$ such that the stalk $\phi_{F,\xi}(\calF)$ has rank-one cohomology for all facets $F$ in $B$, and all other stalks are zero. 
\end{defn}
\begin{figure}[h]
\includegraphics[width=10cm]{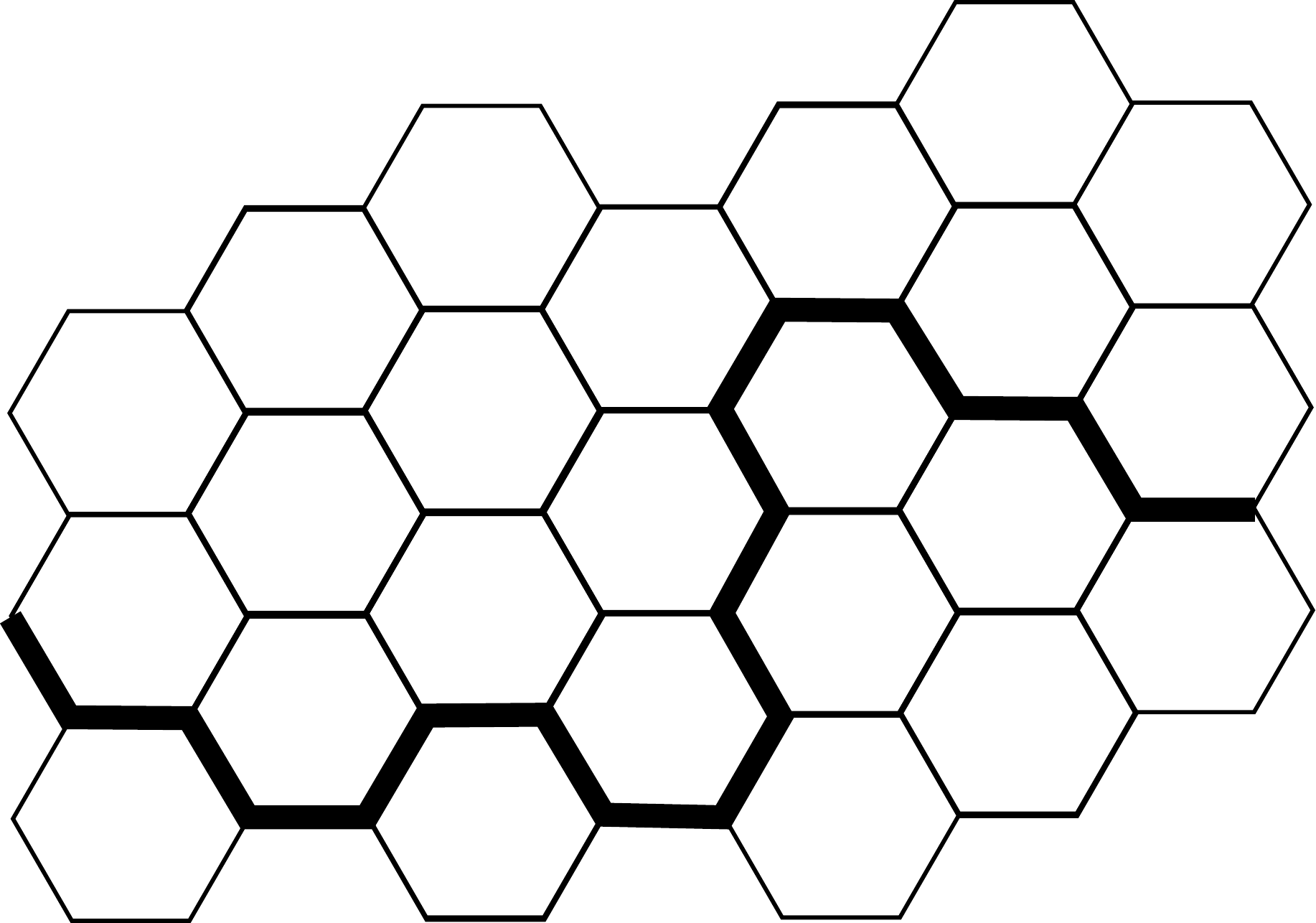}
\caption{Part of the honeycomb $\frH_2,$ with the support (indicated in bold) of a rank-one object in $Q^{inf}_2.$}
\end{figure}

The wrapped category $Q_{n-1}^{wr},$ which is defined as a colimit, is a little to harder to understand directly but admits a very nice set of generators. Note that the stalk functor $\phi_{F,\xi}:Q^{inf}_{n-1}\to(\Perf_k)_{\BZ/2},$ when extended to the cocomplete category $Q^{\Diamond}_{n-1},$ is corepresented by an object $\delta_{F,\xi},$ which we call the \emph{skyscraper} along the facet $F$.

\begin{lemma}
The category $Q^{wr}_{n-1}$ is generated by the set of skyscrapers along facets in $\frH_{n-1}.$
\end{lemma}

The skyscrapers, though defined abstractly, in practice have a simple description. Let $P$ be a permutohedron in $V_n,$ and let $F_i$ be the face shared by $P$ and $P+\lambda_i.$ Let $B$ be the boundary of the region obtained as the union of all positive translates of $P$ in the directions $\{\lambda_1,\ldots,\hat{\lambda}_i,\ldots,\lambda_{n+1}\}.$ See Figure~\ref{fig:sky} for an illustration of such a region $B.$

\begin{prop}[Proposition \ref{prop:skyscrapers} below]The rank-one brane along $B$ is the skyscraper along $F$.
\end{prop}

\subsection{Mirror symmetry}
The mirror to the pair of pants is the Landau-Ginzburg model $(\BA^{n+1},W_{n+1}=z_1\cdots z_{n+1}).$ Let $T^{n+1}$ be the $(n+1)$-torus $(\GG_m)^{n+1}$ and $\TT^n$ the kernel of the map $W_{n+1}:T^{n+1}\to\GG_m.$ Considering the skeleton $\frH_{n-1}$ for the universal abelian cover of the pair of pants instead of the skeleton $\frH_{n-1}/\Lambda_n$ for the pair of pants itself corresponds on the mirror to working equivariantly with respect to the torus $\TT^n$ (whose weight lattice is the lattice $\Lambda_n$). 
Thus, the expectation of homological mirror symmetry is that the wrapped Fukaya category $Q^{wr}_{n-1}$ ought to be equivalent to the torus-equivariant derived category of singularities
\begin{align*}D^b_{sing}(\BA^{n+1},W_{n+1})^{\TT^n}
&=\left(\Coh(W_{n+1}^{-1}(0))/\Perf(W_{n+1}^{-1}(0))\right)^{\TT^n}\\
&=\Coh(W_{n+1}^{-1}(0))^{\TT^n}/\Perf(W_{n+1}^{-1}(0))^{\TT^n},
\end{align*}
since passing to the quotient commutes in this case with taking $\TT^n$ equivariants.
 
There are a couple of ways to understand this category, coming from theorems of Orlov. A presentation of $D^b_{sing}$ which keeps manifest the $\Sigma_{n+1}$ symmetry induced from permutations of the coordinates of $\BA^{n+1}$ is as the category $\MF(\AA^{n+1},W_{n+1})$ of \emph{matrix factorizations} of $W_{n+1}$: the objects of this category are pairs $(V,d),$ for $V$ a $\BZ/2$-graded $k[z_1,\ldots,z_{n+1}]$-module and $d:V\to V$ an odd endomorphism such that $d^2$ is multiplication by $W_{n+1}=z_1\cdots z_nz_{n+1}.$

A less symmetric presentation is given as follows: for a choice of $a\in [n+1],$ let $W_{n+1}^a=W_{n+1}/z_a,$ and set $X^a=(W_{n+1}^a)^{-1}(0).$ Then we have an equivalence of categories
$$
\xymatrix{
D^b_{sing}(\BA^{n+1},W_{n+1})\ar[r]^-{\sim}& D^b(X^a)
}.
$$

This latter presentation is used in \cite{Nwms} to establish a mirror symmetry equivalence. It has the advantage of being built out of a simple inductive definition, but the disadvantage that it breaks the natural $\Sigma_{n+1}$ symmetry of $D^b_{sing}(\BA^{n+1},W_{n+1}).$ In this paper, we use the description as matrix factorizations, which allows us to write a mirror symmetry equivalence which is compatible with more symmetries. Note that the category $D^b_{sing}(\BA^{n+1},W_{n+1})^{\TT^n}$ has a natural action by the group $\tilW_n=\Lambda_n\rtimes\Sigma_{n+1},$ where $\Sigma_{n+1}$ acts by permuting the coordinates and the weight lattice $\Lambda_n$ of $\TT^n$ acts as twists by characters, and this action is manifest in the matrix factorization description. The group $\tilW_n$ also acts on $\frH_{n-1}$ in the obvious way, and hence also on $Q^{wr}_{n-1}.$ The main theorem of this paper is the expected mirror symmetry equivalence between the $A$-model of (the universal abelian cover of) the pair of pants and the (torus-equivariant) $B$-model of $(\BA^{n+1},W_{n+1})$:

\begin{theorem}[Theorem \ref{theorem:main} below]There is an equivalence of categories
$$
\xymatrix{
\Phi:\MF(\AA^{n+1},W_{n+1})^{\TT^n}\ar[r]^-\sim& Q^{wr}_{n-1}
}
$$
which is equivariant for the action of $\tilW_n$ on each side.
\end{theorem}

This equivalence sends the natural generators of the derived category of singularities, the structure sheaves $\cO_n^i$ of the coordinate hyperplanes $\{z_i=0\},$ to the skyscrapers $\delta_{F_i}$ along certain distinguished facets of a standard permutohedron $\frP$ inside $\frH_{n-1}.$ (Specifically, $F_i$ is the facet separating $\frP$ from $\frP+\lambda_i.$) The skyscrapers along other facets correspond to certain complexes formed out of the sheaves $\cO_n^i.$

The relation between the above correspondence and the combinatorics of the permutohedron can be seen more explicitly in the matrix factorization category. Let $[n=1]=\{1,\ldots,n+1\}$; for a nonempty proper subset $I\subset[n+1],$ write $z_I=\prod_{i\in I}z_i$ and $z_{I^c}=\prod_{i\notin I}z_i,$ and similarly set $\lambda_I=\sum_{i\in I}\lambda_i.$ Write $\ul\cO^i_n$ for the image of $\cO^i_n$ in the category $\MF(\BA^{n+1},W_{n+1}).$ Then under the mirror symmetry equivalence $\Phi,$ the skyscraper sheaf along the facet separating $\frP$ and $\frP+\lambda_i$ corresponds to the matrix factorization
$$\ul\cO^i_n=\left(\xymatrix{
k[z_1,\ldots,z_{n+1}]\ar[r]^{z_{\{i\}^c}}&k[z_1,\ldots,z_{n+1}]\ar[r]^{z_{\{i\}}}&k[z_1,\ldots,z_{n+1}]
}\right),$$
while more generally, a skyscraper along the facet separating $\frP$ and $\frP+\lambda_I$ corresponds to the matrix factorization
$$\xymatrix{
k[z_1,\ldots,z_{n+1}]\ar[r]^{z_{I^c}}&k[z_1,\ldots,z_{n+1}]\ar[r]^{z_I}&k[z_1,\ldots,z_{n+1}]
}.$$

\subsection{Acknowledgements}

We would like to thank Yasha Eliashberg, Ailsa Keating, Grisha Mikhalkin, Vivek Shende, Nick Sheridan, Laura Starkston, Alex Takeda, and Ilia Zharkov for helpful conversations. BG is grateful to the NSF for the support of a Graduate Research Fellowship,
and DN  for the support of grant DMS-1502178.

\subsection{Notation and conventions}
We fix an algebraically closed coefficient field $k$ of characteristic zero. Throughout this paper, we work with (usually pretriangulated) differential $\ZZ/2$-graded categories, which we refer to as $\BZ/2$-dg categories. Appropriate homotopical contexts for pretriangulated dg categories have been described in \cite{toen} (with an adaptation of this theory to the $\ZZ/2$-graded case available for instance in \cite{D}, Section 5.1) and \cite{Lur-topos,Lur-algebra} (as $Hk$-linear stable $(\infty,1)$-categories).

We will denote by $\ZZ/2\on{-dgst}_k$ the category of $k$-linear pretriangulated $\ZZ/2$-dg categories, in any of the equivalent homotopical contexts just mentioned. One object of $\ZZ/2\on{-dgst}_k$ we will use often is the $\BZ/2$-dg category $(\Perf_k)_{\BZ/2}$ of $\BZ/2$-dg $k$-modules with finite-dimensional cohomology.

When discussing polyhedral cell complexes in this paper, the word ``face" will mean a general face (of any dimension), while ``facet" will always refer to codimension 1 faces only. Facets of the permutohedron $\frP_n$ are all of the form $\frP_{k-1}\times\frP_{n-k}$ (where $\frP_0=\{pt\},$ and we will refer to facets of the form $\frP_{n-1}\times\frP_0=\frP_{n-1}$ as \emph{maximal facets} of $\frP_n.$

In the table below we collect for the reader's convenience some of the nonstandard or frequently used notations used in this paper, in order of appearance.

\begin{tabular}{l|l}
\hline
$T^{n+1}$ & $(n+1)$-dimensional complex torus\\
$W_{n+1}$ & The map $T^{n+1}\to \GG_m$ given by $(z_1,\ldots,z_{n+1})\mapsto \prod z_i$\\
$\TT^n$ & Ker$(W_{n+1})$\\
$\frP_n$ & $n$-permutohedron\\
$F_I$ & Face of $\frP_n$ corresponding to $I\subset[n+1]$\\
$\Lambda_n$ & Weight lattice of the torus $\TT^n$ \\
$\tilW_n$ & Affine Weyl group $\Lambda_n\rtimes \Sigma_{n+1}$\\
$Q^{inf}_{n-1}$ & Infinitesimally wrapped combinatorial $A$-model category\\
$Q^{wr}_{n-1}$ & Partially wrapped combinatorial $A$-model category\\
$\delta_{F,\xi}$ & Skyscraper along facet $F$ in normal direction $\xi$\\
$B_{P,J}$ & Rank-one brane along $\partial\{P+\sum_{j\in J} n_j\lambda_j\mid n_j\in\NN\}$\\
$\calO^i_n$ & Structure sheaf of $\{z_i=0\}$ in $\Coh(\Spec k[z_1,\ldots,z_{n+1}]/(z_1\ldots z_{n+1}))$\\
$\ul\cO^i_n$ & Image of $\cO^i_n$ under the quotient $\Coh\to\Coh/\Perf$\\
\hline
\end{tabular}


\section{Combinatorial $A$-model}


\subsection{Permutohedron}

Let $T^{n+1}=(\GG_m)^{n+1}$ be the $(n+1)$-dimensional torus, and $W_{n+1}:T^{n+1}\to \GG_m$ the character defined by $W_{n+1}(z_1,\ldots,z_{n+1})= z_1\cdots z_{n+1}.$ We will denote by $\TT^n$ the kernel of $W_{n+1},$ so that we have a short exact sequence of tori
$$
\xymatrix{
1 \ar[r] & \TT^n \ar@{^{(}->}[r] & T^{n+1} \ar@{->>}[r]^{W_{n+1}} & \GG_m\ar[r] & 1.
}
$$

Let $\chi^*(T^{n+1})= \Hom(T^{n+1}, \GG_m) \simeq \BZ^{n+1}$ denote the weight lattice of $T^{n+1}$, and  
let $\lambda_1,\ldots,\lambda_{n+1}$ denote its standard coordinate basis. The above short exact sequence of tori induces a short exact sequence of weight lattices, giving the presentation

$$
\xymatrix{
\chi^*(\TT^n)\simeq \BZ^{n+1}/\BZ\langle \sum_{a=1}^{n+1}\lambda_a\rangle 
}
$$

Throughout this paper, we  will set 
$$
\xymatrix{
\Lambda_n := \chi^*(\TT^n);
&
V_n:=\Lambda_n\otimes\RR.
}
$$ We will abuse notation and write $\lambda_1,\ldots,\lambda_{n+1}$ also for their images in $\Lambda_n$ and $V_n$.

The symmetric group $\Sigma_{n+1}$  naturally acts on $\BZ^{n+1}$ by permutations, and the action descends to 
the quotient $\Lambda_n$ and further to $V_n$. 

\begin{defn}
The  {\em $n$-permutohedron} $\frP_{n} \subset V_n$ is the convex hull
$$
\xymatrix{
\frP_n = \conv\{ \sigma\cdot \left(\frac{1}{n+1}\sum_{a=1}^{n+1}a\lambda_a\right) \in V_n \;|\; \sigma\in \Sigma_{n+1}\}.
}
$$
of the $\Sigma_{n+1}$-orbit of the point $\frac{1}{n+1}\sum_{a=1}^{n+1} a\lambda_a \in V_n$.
\end{defn}

\begin{remark} The above definition of the permutohedron presents it as a convex polytope in the $n$-dimensional quotient space $V_n$  of $\RR^{n+1}.$ This disagrees with the more typical definition as  the convex polytope 
$$
\xymatrix{
 \frP_n'=\conv\{ \sigma_{n+1}\cdot(1,\ldots,n+1)  \in (\BR^{n+1})^* \;|\; \sigma\in \Sigma_{n+1}\}
}
$$
 in the $n$-dimensional  affine subspace 
$$
\xymatrix{
\{(x_1, \ldots, x_{n+1}) \in (\RR^{n+1})^*\mid \sum_{a=1}^{n+1}  x_a = n(n+1)/2\}
}
$$ 

We have chosen our convention with mirror symmetry in mind; in particular,
we prefer a permutohedron which is translated in a natural way by the weight lattice $\Lambda_n = \chi^*(\TT^n)$ rather  than the coweight lattice of $\TT^n.$ 

Our definition agrees with the usual one, up to a duality: after translating $\frP_n'$ by $(\frac{-n}2,\ldots,\frac{-n}2),$ the identification of $\RR^{n+1}$ with its dual space coming from the standard basis of $\RR^{n+1}$ sends $\frP_n'$ to $\frP_n$.
\end{remark}

\begin{example} The 1-permutohedron $\frP_{1}$ is a line segment; the 2-permutohedron $\frP_{2}$ is a hexagon; the 
 3-permutohedron $\frP_{3}$ is a truncated octahedron, with faces consisting of 8 hexagons and 6 squares.  For $n\geq 3$, the permutohedron $\frP_n$ is not a regular polyhedron.
\end{example}

\begin{figure}[h]
\includegraphics[width=4cm]{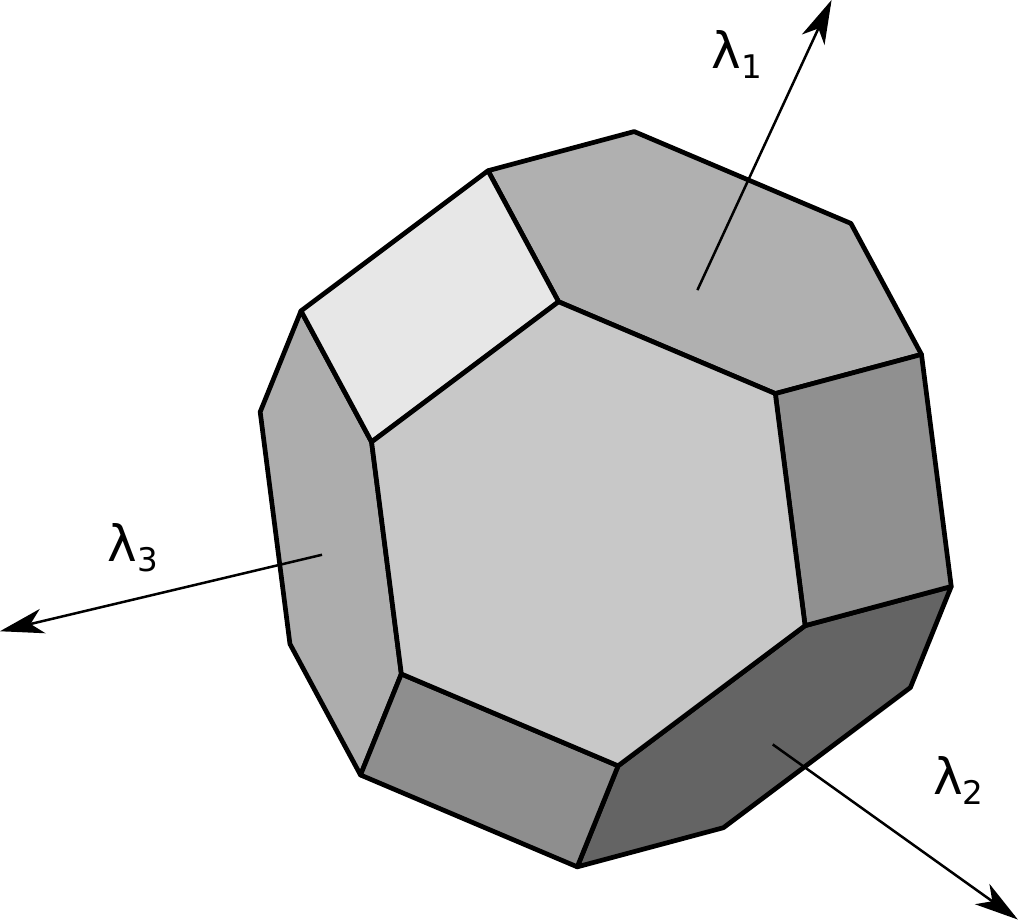}
\caption{The 3-permutohedron $\frP_3$ and three generators of $\Lambda_3$. The fourth generator points directly into the central hexagon.}
\end{figure}

By construction, the symmetric group $\Sigma_{n+1}$ acts transitively on the vertices of the permutohedron.
 To organize the combinatorics of this action, we will find it useful to record here some alternate descriptions and helpful facts about the permutohedron.

 
 \subsubsection{Cayley graph description}
 We first cite from \cite{GG} a description of the permutohedron as a Cayley graph of $\Sigma_{n+1}.$ Recall that the \emph{inversions} in the symmetric group $\Sigma_{n+1}$ are the transpositions of the form $(i \; i+1)$ for some $1\leq i<n+1$.

\begin{lemma}[{\cite[Theorem~1]{GG}}]
The 1-skeleton of the permutohedron $\frP_n$ is the Cayley graph of $\Sigma_{n+1}$ on the generators the inversions in $\Sigma_{n+1}$.
\end{lemma}

This description depends in particular on the choice of a vertex of the permutohedron to correspond to the identity  of $\Sigma_{n+1}$; this vertex will subsequently be denoted $(1)$. Since $\Sigma_{n+1}$ acts transitively on the vertices of the permutohedron, a description of the permutohedron near any vertex is sufficient to understand its global structure.

%
From the Cayley graph perspective, we see that $k$-faces incident to a given vertex of $\frP_n$ correspond precisely to those subgroups of $\Sigma_{n+1}$ generated by $k$ inversions. Call the inversions $(i\; i+1)$ and $(i+1\; i+2)$ \emph{adjacent}. Then if $I\subset \{(1\; 2),(2\; 3),\ldots, (n\; n+1)\}$ is some subset of inversions and we decompose $I=\coprod_i I_i$ into its maximal subsets of adjacent inversions and set $n_i=\#I_i,$ then the $(\sum_{i}n_i)$-dimensional face of $\frP_n$ corresponding to $I$ is of the form $\prod_i \frP_{n_i}.$ In other words:

\begin{cor}
The faces (in every dimension) of the permutohedron $\frP_n$ are products of lower-dimensional permutohedra.
\end{cor}

We can analyze the description above in more detail to attain information about $k$-faces of $\frP_n$ for all $k$; in particular we will be interested in the facets. A facet incident on the vertex $(1)$ is determined by a choice of $n-1$ inversions, and hence the set of such facets is $\{\frP_{n-k}\times \frP_{k-1}\}_{k=1,\ldots,n}.$ The polyhedron $\frP_{n-k}\times \frP_{k-1}$ has $k!(n-k+1)!$ vertices, so the total number of such facets in $\frP_n$ is $\frac{(n+1)!}{k!(n-k+1)!}=\binom{n+1}{k}.$ Adding all of these up, we find that the total number of facets is
$\sum_{k=1}^n\binom{n+1}{k}=2^{n+1}-2.$

Thus the facets of $\frP_n$ are in bijection with proper, nonempty subsets of $\{1,\ldots,n\}.$ In order to transfer the above analysis to the coordinate description of $\frP_n$, it will be helpful to write this bijection in an explicit way.
\begin{lemma}
	Let $S\subset \{1,\ldots,n+1\}$ be a proper, nonempty subset. Define a subset $F_S$ of the vertices of $\frP_n$ by declaring that the vertex $\sigma\cdot (\frac{1}{n+1}\sum_{a=1}^{n+1}a\lambda_a)$ is in $F_S$ if and only if $\sigma(i)<\sigma(j)$ for all pairs $(i\in S,j\notin S).$ Then the map $S\mapsto F_S$ is a bijection between proper nonempty subsets of $\{1,\ldots,n+1\}$ and facets of $\frP_n.$
\end{lemma}
\begin{proof}
Start by analyzing the facets incident on the vertex $v=\sum_{a=1}^{n+1}a\lambda_a,$ which we can treat as the vertex $(1)$ in the Cayley graph. We already have an explicit description for these facets: they correspond to $(n-1)$-element subsets $R\subset \{(1\;2),\ldots,(n\;n+1)\}.$ We claim that the facet corresponding to the subset $R$ which is missing $(i\;i+1)$ has vertex set $F_S,$ for $S=\{1,\ldots,i\}.$ Indeed, the facet corresponding to $R$ is the orbit of $v$ under $\Sigma_{i}\times\Sigma_{n-i+1}\subset\Sigma_{n+1},$ which is precisely $F_S.$

To extend this result to all the facets of the permutohedron, we note that the action of $\Sigma_{n+1}$ on $\{1,\ldots,n+1\}$ induces an action on the set of all proper, nonempty subsets of this set. Similarly, the action of $\Sigma_{n+1}$ on the set of vertices of $\frP_n$ induces an action on the set of facets, and the correspondence $S\mapsto F_S$ is equivariant for these actions (since ultimately both are induced in the same way from the permutation representation of $\Sigma_{n+1}$). Thus, the check we performed at a single vertex is sufficient to prove that $S\mapsto F_S$ is a bijection on the set of all facets of $\frP_n$. 
\end{proof}

\begin{defn} The facets of $\frP_n$ which are of the form $\frP_{n-1}$ will be called the \emph{maximal} facets of the permutohedron $\frP_n.$ Under the above bijection, they correspond to subsets $I\subset \{1,\ldots,n+1\}$ which have either 1 or $n$ elements.
\end{defn}

\subsubsection{Minkowski sum description}
 The other useful description which $\frP_n$ admits is as a Minkowski sum of line segments. Recall that the \emph{Minkowski sum} of two subsets $A$ and $B$ of $\RR^n$ is 
$$
\xymatrix{
A+B:=\{a+b\mid a\in A,b\in B\}.
}
$$
Let $\lambda_1,\ldots,\lambda_{n+1}$ be the standard basis vectors of $\RR^{n+1}.$ Then we have the following description of $\frP_n$:
\begin{lemma}
The $n$-permutohedron $\frP_n$ can be represented as the Minkowski sum
$$
\xymatrix{
\sum_{1\leq i<j\leq n+1}[\frac{\lambda_i-\lambda_j}{2(n+1)},\frac{\lambda_j-\lambda_i}{2(n+1)}],
}
$$
where $[a,b]$ denotes the set $\{ta+(1-t)b\mid 0\leq t\leq 1\},$ for $a,b\in V_n.$
\end{lemma}
\begin{proof}
By construction, the set of vertices of the permutohedron is contained in the Minkowski sum
$$\sum_{1\leq i<j\leq n+1}\left\{\frac{\lambda_i-\lambda_j}{2(n+1)},\frac{\lambda_j-\lambda_i}{2(n+1)}\right\}$$
of two-element sets, and it is a general fact that the convex hull $conv(A+B)$ of a Minkowski sum $A+B$ is equal to the Minkowski sum $conv(A)+conv(B).$
\end{proof}

A Minkowski sum of line segments is also known as a \emph{zonotope}. As a cube is also a Minkowski sum of line segments, a zonotope can also be understood as the projection of a cube under an affine transformation. Hence, for instance, $\frP_n$ is an affine projection of the $\binom{n+1}{2}$-dimensional cube. Zonotopes have many nice properties, and the zonotopal perspective is often helpful for inductively describing the geometry of $\frP_n$; many of the combinatorial arguments which we made above could have proceeded in the language of zonotopes.

\subsubsection{Voronoi cell description}

In fact, $\frP_n$ is a special kind of zonotope. Recall that every rank $n$ lattice $\Lambda\subset \RR^n$ has an associated \emph{Voronoi tiling}, a tiling of $\RR^n$ symmetric under translation by $\Lambda$: for every lattice point $x\in \Lambda,$ there is a Voronoi cell $R_x$ centered at $x,$ where we define
$$
\xymatrix{
R_x:=\{y\in\RR^n\mid |y-x|<|y-x'|\text{ for all }x\neq x'\in \Lambda\}.
}
$$

The vector space $V_n$ is a quotient of $\RR^{n+1}$ and hence a subspace of its dual $(\RR^{n+1})^*,$ from which it inherits the standard metric (coming from the dual basis to $\lambda_1,\ldots,\lambda_{n+1}$). With respect to this metric, permutohedra are Voronoi cells:

\begin{lemma}[\cite{CS}]
The $n$-permutohedron $\frP_n$ is a Voronoi cell for the rank $n$ lattice $\Lambda_n\subset V_n.$
\end{lemma}

Let $\tilW_n$ denote the semidirect product $\Lambda_n \rtimes \Sigma_{n+1}$. Then we have the following corollary:

\begin{cor} 
$\Lambda_n$-translates of $\frP_n$ provide a tessellation of $V_n$
preserved by the natural $\tilW_n$-action.
\end{cor}

This means in particular that for every facet $F$ of $\frP_n,$ there is a vector $v\in \Lambda_n$ such that $\frac12v$ is the center of $F$. 
Recall that earlier we exhibited a bijection between facets of $\frP_n$ and nonempty proper subsets $I\subset \{1,\ldots, n+1\}.$ Now we see another way to understand this bijection:
let $\lambda_I=\sum_{i\in I}\lambda_i.$ Then $\Lambda_n$ is the $\NN$-span of $\{\lambda_I\}_{\emptyset\subsetneq I\subsetneq [n+1]}.$ The Voronoi cells adjacent to the cell at the origin are those cells which are centered at points of the form $\lambda_I.$ This gives a correspondence between facets of $\frP_n$ and nonempty proper subsets $I$ of $[n+1],$ associating to $I$ the facet $F_I$ through which $\lambda_I$ points.
Moreover, if we write $I^c$ for the complement of $I$ in $[n-1]$ (also nonempty and proper because $I$ is), then $\lambda_{I^c}=-\lambda_I,$ so $F_{I^c}$ is the face opposite to $F_I$.


\subsection{Honeycomb}

We are now ready to introduce the main object of study in this paper.

\begin{defn}
Let $ \partial \frP_n \subset V_n$ denote the boundary of the $n$-permutohedron.

The {\em $(n-1)$-honeycomb} $\frH_{n-1}\subset V_n$ is the piecewise linear hypersurface given by the union of translates 
$$
\xymatrix{
\frH_{n-1} = \partial \frP_n + \Lambda_n
}
$$
\end{defn}


We will describe the singularities of $\frH_{n-1}$  in the language of \cite{Narb}, which introduced a special class of Legendrian singularities, called \emph{arboreal singularities}; these are conjecturally the deformation-stable Legendrian singularities, and the category of microlocal sheaves along an arboreal singularity admits a simple description as modules over an acyclic quiver. In \cite{Nexp}, it was shown that any Legendrian singularity admits a deformation to an arboreal singularity which preserves the category of microlocal sheaves along the singularity. Thus the study of microlocal sheaves along any Lagrangian in a symplectic manifold is reduced to a two-step process: first, deform the Lagrangian so that each singularity is the cone on some arboreal singularity, then glue the corresponding categories together.

In the case of the honeycomb (when, in section 4, we describe its appearance as a Weinstein skeleton), the first step is unnecessary, since, as we will see, all the singularities of the honeycomb are already arboreal. In fact, in the language of \cite{Narb}, all its singularities are arboreal singularities of type $A_k$ for some $k$; this arboreal singularity is homeomorphic to the cone on the $(k-2)$ skeleton of the $k$-simplex $\Delta^k.$ This structure makes it possible to describe the symplectic geometry of the honeycomb by means of the combinatorial constructions in this section.

Below we describe the singularities of the honeycomb $\frH_{n-1}$ and show that they are all stratified homeomorphic to $A_n$ arboreal singularities; later, when we present the honeycomb as a Lagrangian skeleton for the pair of pants, we will see that the singularities of this Lagrangian actually have the symplectic geometry of $A_n$ arboreal singularities.

Let $\Delta^k \subset \BR^{k+1}$ be a $k$-simplex with barycenter at the origin,
and let $sk_{k-2}\Delta^k \subset \BR^{k+1}$ denote its $(k-2)$-skeleton.
For $A\subset \BR^{k+1}$, we denote by $\Cone(A)$ the cone $\BR_{\geq 0} \cdot A\subset \BR^{k+1}$ given by scalings of $A$ (with cone point at the origin).

\begin{defn}
 
For $k> 0$, the topological $A_k$-hypersurface singularity is given by
$$
\xymatrix{
L_{A_k} = \Cone(sk_{k-2}\Delta^k) \subset \BR^{k+1}.
}
$$
\end{defn}

This will be our topological model for the Legendrian arboreal singularity $\calL_{A_n}$; we will return to the symplectic geometry of this singularity in the proof of Proposition \ref{prop:miccat}. Since so far $\frH_{n-1}$ is only a topological space, we will describe its singularities for now in terms of the topological singularity $L_{A_k}.$

\begin{defn} Define the stratified space $\frH_{n-1}$ to be the space $\frH_{n-1}$ equipped with the stratification by relatively open faces. In other words, the strata of $\frH_{n-1}$ are indexed by faces $F$, and the stratum $S_F$ corresponding to the face $F$ is just the face $F$, not including any lower-dimensional faces incident on $F$.
\end{defn} 
Now we can describe precisely the singularities of the honeycomb:

\begin{prop}\label{prop:sings}
If $p\in\frH_{n-1}$ is a point in an $(n-k)$-dimensional stratum of the honeycomb, then a neighborhood of $p$ in $\frH_{n-1}$ is homeomorphic to $L_{A_k}\times\RR^{n-k}.$
\end{prop}
\begin{proof}
Suppose first that $p$ is in the 0-dimensional stratum of $\frH_{n-1},$ \emph{i.e.}, $p$ is a vertex in the honeycomb. In this case the proposition claims that in a small neighborhood $B$ of $p$, the complement $B^\circ:=B\setminus\{p\}$ is homeomorphic to $\RR_{>0}\times sk_{n-2}\Delta^n.$ As a simplicial complex, $sk_{n-2}\Delta^n$ is determined by its face poset, which is the poset of nonempty subsets of $[n+1]:=\{0,1,\ldots,n\}$ containing at most $n-1$ elements.

Let $\epsilon>0$ such that the radius $\epsilon$ sphere $S_\epsilon$ centered at $p$ is contained in $B$, and let $B_{link}$ be the intersection of $B^\circ$ with $S_\epsilon.$ Then positive dilation gives a homeomorphism $\RR_{>0}\times B_{link}\cong B^\circ.$ Hence we need to prove that $B_{link}$ admits the structure of a regular cell complex with face poset isomorphic to the poset of nonempty subsets of $[n+1]$ containing at most $n-1$ elements.

We claim first that in the honeycomb, the vertex $p$ is incident on $n+1$ edges. To see this, we recall the Cayley graph description of $\frP_n$: A vertex $v$ in $\frP_n$ is incident on $n$ edges, and the $n$ $\Lambda_n$-translates of $\frP_n$ which contain $v$ correspond to the $n$ facets in $\frP_n$ containing $v$, which correspond in turn to the $n$ choices of $n-1$ edges in $\frP_n$ which contain $v.$ A copy of $\frP_n$ which contains $v$ is determined by $n$ edges containing $v$, so each translate incident on $v$ contains exactly one new edge which contains $v$. In fact, this edge is the same for all translates: otherwise, translations would produce at least two new edges containing $v$, but no translate could contain both of these, contradicting the fact that translates of $\frP_n$ tile space.

Now by symmetry we can conclude that for $1\leq k\leq n-1,$ any choice of $k$ edges containing $p$ determines a $k$-face containing $p$ in some translate of $\partial\frP_n$, and these are all the faces containing $p$. In other words, there is a bijection 
$$
\xymatrix{
\{\text{$k$-faces in $\frH_{n-1}$ containing $p$}\}\cong \{\text{$k$-element subsets of $[n+1]$}\},
}
$$
and incidence relations among these are given by the natural poset structure on the set of subsets of $[n+1].$

Thus we have established the proposition in the case where $p$ is in a 0-dimensional stratum. We can derive the case where $p$ is in a $k$-dimensional stratum by starting with $p'$ a vertex contained in a small neighborhood $B'$, and then restricting $B'$ to a ball $B$ which does not contain any strata of dimension less than $k$. In the analysis above, this corresponds to restricting to a subposet of the set of subsets of $[n+1],$ which we can identify as the face poset of $sk_{n-k-2}\Delta^{n-k}.$
\end{proof}

The proof above actually establishes more than an abstract description of the singularities of the honeycomb $\frH_{n-1}$: it also explains the inductive way in which they are embedded in one another. Note that the $A_k$ singularity $L_{A_k}= \Cone(sk_{k-2}\Delta^k)$ contains $k+1$ copies of $\RR\times L_{A_{k-1}},$ each embedded as the cone on a small neighborhood of a vertex in $sk_{k-2}\Delta^k.$  Since the description we gave above respects all of these identifications, we can elaborate on the above proposition:
\begin{cor}
	For $0\leq k\leq n-2,$ let $\alpha$ be a $k$-face in $\frH_{n-1}$ incident on a $(k+1)$-face $\beta.$ Then the singularity $L_{A_{n-k}}\times \RR^k$ lying along $\beta$ is one of the $n-k+1$ copies of $L_{A_{n-k-1}}\times\RR^{k+1}$ embedded as described above in the singularity $L_{A_{n-k}}\times\RR^{k}$ lying along $\alpha.$
\end{cor}


\subsection{Cyclic structure sheaf}

From the above description we see that the only data needed around a point in $\frH_{n-1}$ to determine its singularity type is the number of permutohedra in the tiling of $V_n$ which contain that point. 
We will encode that data, along with the data of relations among these singularities, in a cosheaf on $\frH_{n-1},$ from which we will subsequently produce a combinatorial model for the $A$-model category associated to $\frH_{n-1}.$

\begin{defn}
Define the cosheaf of finite sets $\calO_{n-1}$ over the honeycomb $\frH_{n-1}$ to be the connected components
of the complement
$$
\xymatrix{
\calO_{n-1}(B) = \pi_0(B \setminus (B\cap \frH_{n-1})) 
}
$$
for small open balls $B\subset V_n$ centered at points of $\frH_{n-1}.$ For inclusions $\iota:B'\hookrightarrow B,$ the corresponding corestriction map is the map induced on $\pi_0$ by $\iota.$
\end{defn}

Because the cosheaf $\calO_{n-1}$ is constructible with respect to a stratification of $\frH_{n-1}$ as a regular cell complex, the cosheaf $\calO_{n-1}$ can be equivalently described as a contravariant functor from the exit-path category $P(\frH_{n-1})$ associated to the stratified space $\frH_{n-1}.$ (Likewise, a constructible sheaf on $\frH_{n-1}$ is a covariant functor from $P(\frH_{n-1}).$) The category $P(\frH_{n-1})$ is equivalent to the poset which has one point $\alpha$ for each stratum $S_\alpha$ in $\frH_{n-1}$, and one arrow $\alpha\to\beta$ for every relation $S_\alpha\subset \bar{S}_\beta.$

For each $\alpha\in P(\frH_{n-1})$, pick a ball $U_\alpha$ such that $U_\alpha\cap\bar{S}_\alpha=S_\alpha,$ and $U_\alpha\cap S_\beta=\emptyset$ for $\alpha,\beta$ incomparable in $P(\frH_{n-1})$. Then we can define $\calO_{n-1}$ as the functor

$$
\xymatrix{
\calO_{n-1}:P(\frH_{n-1})^{op} \to \on{Sets}, 
&
\calO_{n-1}(\alpha)=\pi_0(U_\alpha\setminus (U_\alpha\cap \frH_{n-1})),
}
$$
where the map $\pi_0(U_\beta\setminus (U_\beta\cap \frH_{n-1}))\to \pi_0(U_\alpha\setminus (U_\alpha\cap \frH_{n-1}))$ induced by the incidence $\alpha\to\beta$ is defined through the inclusions
$$
\xymatrix{
U_\alpha\setminus (U_\alpha\cap \frH_{n-1})& \ar@{_{(}->}[l] U_\alpha\cap U_\beta\setminus (U_\alpha\cap U_\beta\cap \frH_{n-1}) \ar@{^{(}->}[r]^-\sim & U_\beta\setminus (U_\beta\cap \frH_{n-1}),
}
$$ using that the second is a homotopy equivalence.

Recall the cyclic category $\Lambda$ of finite cyclically ordered nonempty sets: its objects are finite subsets $S\subset S^1$, and morphisms $S\to S'$ are given by homotopy classes of degree 1 maps $\varphi:S^1\to S^1$ such that $\varphi(S)\subset S'.$ We would like to lift $\calO_{n-1}$ to $\Lambda$; in other words, we want to express $\calO_{n-1}$ as the composition of the forgetful functor $\Lambda\to \on{Sets}$ with a functor $\tilcalO_{n-1}:P(\frH_{n-1})\to\Lambda.$ Such a lift is the same as a choice of cyclic ordering on every set $\calO_{n-1}(\alpha).$

Moreover, we want this lift to respect the $\tilW_n$ symmetry of $\frH_{n-1}.$ The $\tilW_n$ symmetry of the honeycomb $\frH_{n-1}$ induces a $\tilW_n$ action on $P(\frH_{n-1})$, and for $\alpha\in P(\frH_{n-1}),$ this symmetry also induces an action of $\tilW_n$ on the set $\pi_0(B_\alpha\setminus (B_\alpha\cap \frH_{n-1})).$ This action does not affect the set itself but will alter the cyclic ordering if this set is endowed with one. Hence the condition of $\tilW_n$-equivariance places extra requirements on the structure of $\tilcalO_{n-1}.$

\begin{lemma}
There are $n!$ possible $\tilW_n$-equivariant choices of lift $\tilcalO_{n-1},$ each determined by a choice of cyclic ordering on $\calO_{n-1}(\alpha),$ where $\alpha$ corresponds to any vertex in $\frH_{n-1}.$
\end{lemma}
\begin{proof}
Suppose we have chosen a cyclic ordering on $\calO_{n-1}(\alpha)$ as in the lemma. For $\alpha\to \beta,$ the inclusion $\calO_{n-1}(\beta)\hookrightarrow\calO_{n-1}(\alpha)$ determines a cyclic ordering on all $\calO_{n-1}(\beta).$ Conversely, if $\alpha'\in P(\frH_{n-1})$ also corresponds to a 0-dimensional stratum in $\frH_{n-1},$ then the action of $\tilW_n$ transfers the cyclic ordering on $\calO_{n-1}(\alpha)$ to $\calO_{n-1}(\alpha')$ and hence also determines a cyclic order on all $\beta$ with $\alpha'\to\beta.$ We have to show that for any incidence relation of the form
$$
\xymatrix{
\alpha\ar[r]&\beta&\ar[l]\alpha'
}
$$
in $P(\frH_{n-1}),$ both of the above methods of determining a cyclic order on $\calO_{n-1}(\beta)$ coincide. In other words, if $\sigma\in \tilW_n$ is any element taking $\alpha$ to $\alpha'$ and taking $\beta$ to itself, then we must show that $\sigma$ acts as the trivial permutation on $\calO_{n-1}(\beta).$ Since the affine Weyl group $\tilW_n$ is generated by reflections through root hyperplanes, we may assume $\sigma$ is a reflection through a root hyperplane. Any intersection of such a hyperplane with a face of a Voronoi cell for the lattice $\Lambda_n$ is transverse. Thus, if $\sigma\cdot \beta=\beta,$ then $\sigma$ is reflection through a hyperplane intersecting every connected component in $\calO_{n-1}(\beta)$ and hence acts as the trivial permutation on $\calO_{n-1}(\beta).$
\end{proof}

Now we fix a cyclic order at vertices as follows: at a vertex $\alpha$, elements of the set $\calO_{n-1}(\alpha)$ can be identified with the $n+1$ copies of the permutohedron which contain $\alpha.$ We endow this set of permutohedra with the cyclic order $[P_0,\ldots,P_n]$ such that (taking indices cyclically modulo $n+1$) we have $P_i=P_{i-1}+\lambda_i.$ Since we have $w\cdot P_i=w\cdot P_{i-1}+w\cdot\lambda_i$ for any $w\in \tilW_n,$ this gives a consistent choice of cyclic structure at all vertices. 

\begin{figure}[h]
\includegraphics[width=4cm]{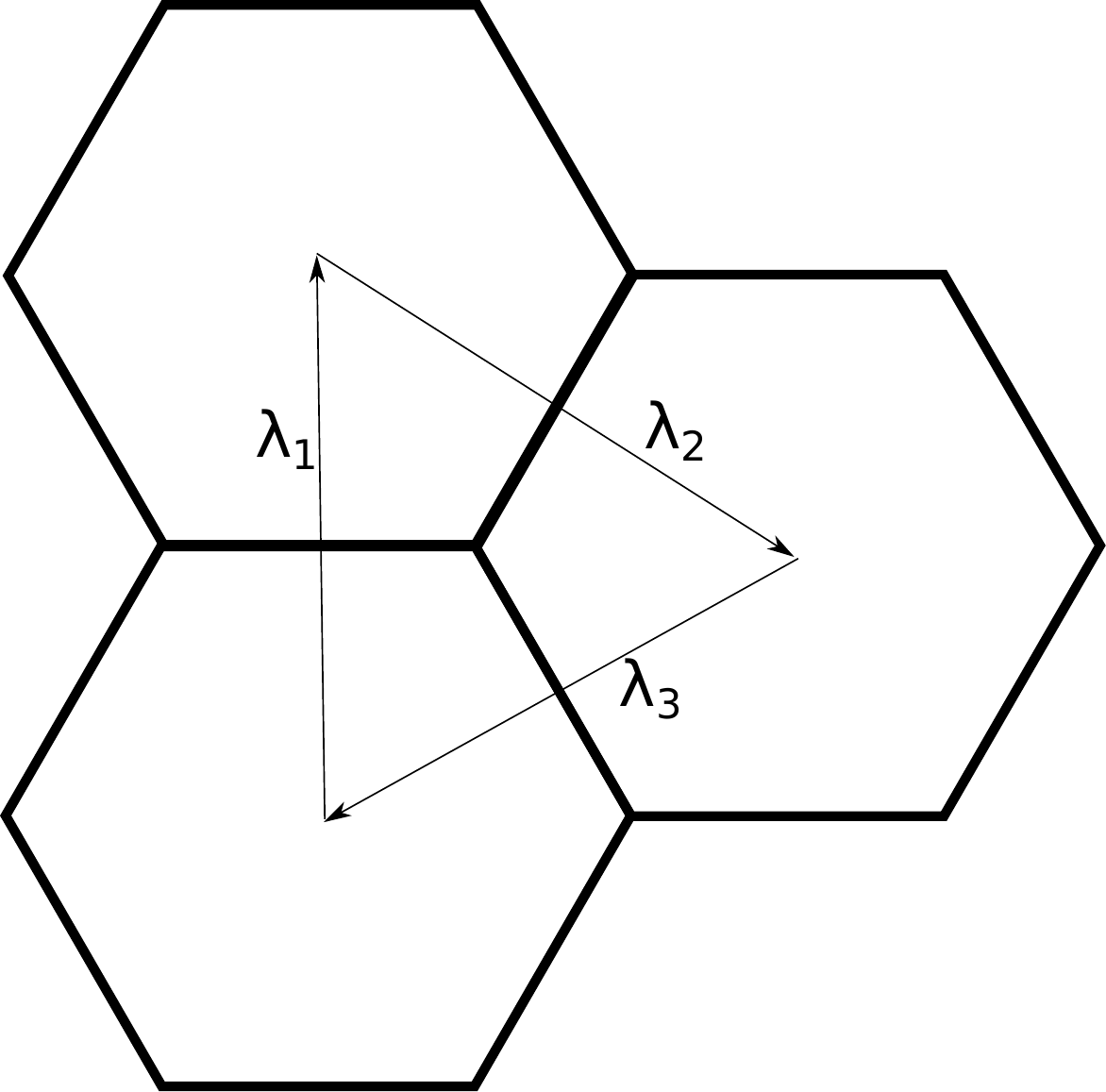}
\caption{The cyclic order at a vertex in $\frH_1$}
\end{figure}

\begin{defn}\label{definition:cyc}
We define the functor
$$
\xymatrix{
\widetilde{\calO}_{n-1}:P(\frH_{n-1})^{op} \to \Lambda
}
$$
by using this cyclic order to lift the cosheaf of sets defined above to a cosheaf of cyclic sets:
$$
\xymatrix{
\widetilde{\calO}_{n-1}(\alpha)=\pi_0(B_\alpha\setminus (B_\alpha\cap \frH_{n-1})),
}
$$
where $\widetilde{\calO}_{n-1}(\alpha)$ is given the cyclic order described in the previous paragraph. This functor factors through the non-full subcategory $\Lambda_{inj}$ of cyclic sets and injective morphisms, and we will denote the resulting functor $P(\frH_{n-1})^{op}\to\Lambda_{inj}$ also by $\tilcalO_{n-1}.$

\end{defn}


\subsection{Quantization} The cyclic cosheaf $\tilcalO_{n-1}$ encodes the data of all the singularities of $\frH_{n-1},$ our combinatorial model for a skeleton of the pair of pants. Following the procedure described in \cite{Ncyc}, we can produce from this cosheaf a sheaf (respectively, cosheaf) on $\frH_{n-1}$ whose global sections are a dg category modeling the infinitesimally wrapped (respectively, partially wrapped) Fukaya category of branes running along the skeleton $\frH_{n-1}$. This procedure is analogous to the constructions of topological Fukaya categories described in \cite{DK, HKK}, although thanks to the arboreal singularities of our skeleton, the construction we describe here works in arbitrary dimensions.

The key ingredient in our construction is a functor
$$
\xymatrix{
\calQ: \Lambda_{inj}^{op}\ar[r] &  \BZ/2\on{-dgst}_k
}
$$
which is described as Construction~\ref{cons:cycfunc} below. First, for $S=[s_1,\ldots,s_{n+1}]$ a cyclic set of $n+1$ elements, consider the $\BZ/2$-dg category $(A_n\on{-Perf}_k)_{\BZ/2},$ whose objects include the $n$ simple modules $k_1,\ldots,k_n$ and the shifted injective-projective $I_n[1]=P_1[1].$ We will relabel these objects $s_1,\ldots,s_{n+1},$ respectively, and denote by $\langle s_1,\ldots,s_{n+1}\rangle$ the full subcategory on these objects. Let $\calC_{S}$ be the dg category of twisted complexes on $\langle s_1,\ldots,s_{n+1}\rangle$:
$$
\calC_S:=\on{Tw}\langle s_1,\ldots,s_{n+1}\rangle.
$$

Since the category $(A_n\on{-Perf}_k)_{\BZ/2}$ is generated by simples $s_i,$ the category $\calC_{S}$ is equivalent to $(A_n\on{-Perf}_k)_{\BZ/2}$ but with a manifest cyclic symmetry: the category $\calC_S$ admits an action of $\ZZ/(n+1)\ZZ$, whose generator takes $s_i$ to $s_{i+1},$ indexed cyclically. 
\begin{cons}[\cite{Ncyc} Proposition 3.5] \label{cons:cycfunc}The functor
$$\xymatrix{
\calQ: \Lambda^{op}\ar[r] & \BZ/2\on{-dgst}_k
}
$$
has value $\calQ(S)=\calC_{S}$, and the map $\calQ(i):\calC_{S'}\to\calC_S$ induced by the inclusion $i:S\hookrightarrow S'$ is the dg quotient of $\calC_{S'}$ by the full subcategory on $\{s_i\mid s_i\in S'\setminus S\}.$
\end{cons}
\begin{remark}The functor described in \cite{Ncyc} actually has target in the category of (2-periodic) $A_\infty$ categories and strict functors; the functor described here is a ($\BZ/2$)-dg model of that one. (In fact, below we will describe two different $\BZ/2$-dg models of this functor.)
	See also \cite{DK} for a more extensive discussion of this functor, modeled there using the category of matrix factorizations of $x^n.$
\end{remark}

\begin{remark}\label{remark:Sindex}
	The notational confusion of $s_i\in S$ with $s_i$ the element of $\calC_S$ in the lemma above is meant to indicate that our set of distinguished generators of $\calC_S$ is indexed by the cyclic set $S$. The cyclic sets we consider will in general be sets $S=[P_1,\ldots,P_{n+1}]$ of adjacent permutohedra as at the end of the previous section; in this case we will continue to denote the generators of $\calC_S$ by $s_1,\ldots,s_{n+1},$ with the understanding that $s_i$ is indexed to $P_i.$ We will always understand the indexing of the $s_i$ cyclically, so that, for instance, we may denote $s_{n+1}$ also by $s_{0}.$
\end{remark}

A choice of a linear order $\{s_1\to\cdots s_n\to s_{n+1}\}$ underlying the cyclic order on $S$ picks out an equivalence $\calC_S\cong (A_n\on{-Perf}_k)_{\BZ/2}$ sending $s_i$ to the simple object $k_i$ for $i=1,\ldots,n$ and sending $s_{n+1}$ to $I_n[1]\cong P_1[1].$ But since the cyclic set $S$ does not have a distinguished linear order, there is no distinguished equivalence $\calC_S\cong (A_n\on{-Perf}_k)_{\BZ/2}$ without making such a choice.

To see a more explicit description of the maps $\calC_{S'}\to\calC_S$ which $\cQ$ induces from an inclusion $S\to S',$ note first that the category $\calC_S$ is generated by degree 1 morphisms $\alpha_i:s_i\to s_{i+1}$. (This corresponds in $A_{n+1}\on{-Perf}_k$ to the degree 1 map of simple objects $k_i\to k_{i+1}$ representing the class of the nontrivial extension.) 
For $|j-i|<n,$ we can form a complex $s_{ij}$ from the objects $s_i,s_{i+1},\ldots,s_j$ by taking successive extensions by the maps $\alpha_i$. We write this object schematically as a complex
$$\xymatrix{
s_{i,j}:=(s_i\ar[r]^-{\alpha_i}&s_{i+1}\ar[r]^-{\alpha_{i+1}}&\cdots\ar[r]^-{\alpha_{j-1}}&s_j)
},$$
where $s_j$ is placed in degree 0.
Under the equivalence $\calC_S\cong (A_n\on{-Perf}_k)_{\BZ/2}$ which sends $s_i$ to $k_1,$ the object $s_{ij}$ corresponds to the 
$A_n$-representation
\[
\xymatrix{
	k\ar[r]^-\sim&k\ar[r]^-\sim&\cdots\ar[r]^-\sim&k\ar[r]&0\ar[r]&\cdots\ar[r]&0
}
\]
with $j-i+1$ nonzero terms.

Let $i_k$ be the inclusion $\{1,\ldots,n+1\}\setminus \{k\}\hookrightarrow \{1,\ldots,n+1\}.$ Then the map $\cQ(i_k)$ acts as
$$
\cQ(i_k)(s_{i,j})=
\begin{cases}
s_{i,j}&i\neq k\neq j\\
s_{i+1,j}&i=k\\
s_{i,j-1}&j=k.
\end{cases}
$$

Since any object in $\calC_S$ is a direct sum of the $s_{i,j}$ and any inclusion $S\to S'$ can be written as a composition of inclusions which miss one element, the above gives a complete description of the behavior of the functor $\calQ$ on the subcategory of $\Lambda^{op}$ whose morphisms are injections of cyclic sets.

It will also be useful to have one other description of the functor $\calQ$ which will give us a different way of thinking about the Fukaya category we describe below. Note that instead of taking the pretriangulated closure of $\langle s_0,\ldots,s_n\rangle$ by using twisted complexes, we could equally well have used perfect modules: \emph{i.e.}, we have an equivalence
$$\calC_S\cong\langle s_0,\ldots,s_n\rangle\on{-(Perf}_k)_{\ZZ/2}:= \Fun^{ex}(\langle s_0,\ldots,s_n\rangle^{op},(\Perf_k)_{\BZ/2}).$$

Since the category $\langle s_0,\ldots,s_n\rangle$ is generated by the degree 1 maps $\alpha_i,$ an object $\calF$ in $\langle s_0,\ldots,s_n\rangle\on{-(Perf}_k)_{\ZZ/2}$ is determined by the $n+1$ objects $\calF(s_i)$ of $(\Perf_k)_{\BZ/2}$ and the $n+1$ maps $\calF(s_i)\leftarrow\calF(s_{i+1}):\calF(\alpha_i).$

The $n+1$ equivalences of this category with $(A_n\on{-Perf}_k)_{\BZ/2}$ come from cyclic reindexing and then applying the equivalence
$$
\xymatrix{
(\langle s_0,\ldots,s_n\rangle\on{-Perf}_k)_{\BZ/2}\ar[r]&(\langle s_1,\ldots,s_n\rangle\on{-Perf}_k)_{\BZ/2}
}
$$
given by forgetting $\calF(s_0)$ and the maps $\calF(s_0)\leftarrow \calF(s_1)$ and $\calF(s_n)\leftarrow\calF(s_0).$

The functor $\calQ$ is defined in this language by

$$
(\cQ(i_k))(\calF)=\left[s_i\mapsto
\begin{cases}
	(\calF(s_{k-1})\leftarrow\calF(s_k))&i=k-1\text{ (mod $n+1$)}\\
\calF(s_i)&i\neq k-1
\end{cases}\right].
$$

\begin{remark}We will see below that the two dg models for the functor $\cQ$, using twisted complexes or using perfect modules, give two different ways of talking about the Fukaya categories we construct. The first is adapted to describing the support of a brane along an arboreal Lagrangian, while the second is better for describing its transverse geometry. (See also Example~\ref{ex:braneconfigs} and the preceding discussion.)
\end{remark}

In addition to the functor $\calQ,$ we would like to produce a \emph{covariant} functor
$$
\xymatrix{
\calQ^{wr}:\Lambda_{inj}\ar[r] & \BZ/2\on{-dgst}_k,
}
$$
in order to produce a \emph{cosheaf} of dg categories on $\frH_{n-1}.$ 
For any object $S$ in $\Lambda$, and for any map $i:S\to S'$ in $\Lambda_{inj}$, we define
$$
\xymatrix{
\cQ^{wr}(S):=\cQ(S),&\cQ^{wr}(i):=\cQ(i)^L:\cQ^{wr}(S)\to\cQ^{wr}(S'),
}
$$
where we write $\cQ(i)^L$ for the left adjoint to the map $\cQ(i).$

\begin{defn}
1) Define the {\em local wrapped and infinitesimal quantizations} $\calQ^{wr}_{n-1}$ and $\calQ^{inf}_{n-1}$ to be the respective compositions
$$
\xymatrix{
\calQ^{wr}_{n-1} = \calQ^{wr}\circ \tilcalO_{n-1}: P(\frH_{n-1})^{op} \ar[r] & \BZ/2\on{-dgst}_k, \\
\calQ^{inf}_{n-1} = \calQ\circ \tilcalO_{n-1}^{op}:P(\frH_{n-1}) \ar[r]&\BZ/2\on{-dgst}_k
}
$$
These are, respectively, a cosheaf and sheaf of $\BZ/2$-dg-categories on $\frH_{n-1}$ whose sections in a small ball around a point in a $k$-face are equivalent to the $\BZ/2$-dg-category of representations of the $A_{n-k}$ quiver.

2) Define the {\em global infinitesimal quantization} $Q^{inf}_{n-1}$ to be the global sections of the sheaf $\calQ_{n-1}$:
$$
\xymatrix{
Q^{inf}_{n-1} = \on{lim}_{P(\frH_{n-1})} \calQ_{n-1}.
}
$$
Define the \emph{global wrapped quantization} to be the idempotent-completion of the global sections of the cosheaf $\calQ_{n-1}^{wr}$:
$$
\xymatrix{
Q^{wr}_{n-1} = \on{Idem}\left(\on{colim}_{P(\frH_{n-1})^{op}} \calQ^{wr}_{n-1}\right).
}
$$
\end{defn}

These will be our respective models of the infinitesimal and (idempotent-completed) wrapped Fukaya categories of the $\ZZ^n$ cover $\tilP_{n-1}$ of the $(n-1)$-dimensional pair of pants.

\begin{remark}These categories inherit $\tilW_n$ symmetries from the $\tilW_n$-action on the poset $\frH_{n-1}$ (and the equivariance of the cyclic structure sheaf $\tilcalO_{n-1}$).
In particular, these categories have an action by the normal subgroup $\Lambda_n\subset \tilW_n$ of translations. We will denote the action of a translation $\lambda\in \Lambda_n$ on an object $\calF$ by $\calF\langle\lambda\rangle,$ so that
$\calF\langle\lambda\rangle(U):=\calF(U+\lambda),$
to match with our notation on the B-side in the next section.
\end{remark}

\subsection{The quantization categories}
We would like to describe more explicitly the categories $Q_{n-1}^{inf}$ and $Q_{n-1}^{wr}.$ Since the category $Q^{inf}_{n-1}$ is presented as a limit, it is easier to understand: an object of $Q^{inf}_{n-1}$ is specified by the data of an object in the categories $\calQ^{inf}_{n-1}(\alpha)$ associated to each vertex $\alpha$ in $\frH_{n-1}$ and coherent isomorphisms relating the results of restriction maps $\calQ^{inf}_{n-1}(\alpha_i)\to\calQ^{inf}_{n-1}(F)$ associated to pairs of inclusions $\alpha_1,\alpha_2\hookrightarrow \bar{F}$ from faces $\alpha_1,\alpha_2$ into the closure of a higher-dimensional face $F$. For $\calF$ an object of $Q^{inf}_{n-1}$ and $\alpha$ a face in the honeycomb $\frH_{n-1},$ we will denote by $\calF_\alpha \in \calQ^{inf}(\alpha)$ the component of $\calF$ placed at the face $\alpha.$

There are two useful ways to understand this category, corresponding to the two descriptions of $\calQ^{inf}(\alpha)=\calC_{\tilcalO(\alpha)},$ as perfect modules and as twisted complexes. We will begin with the first perspective, which allows us to think of an object in $\calQ^{inf}_{n-1}$ as the data of an object of $(\Perf_k)_{\BZ/2}$ at each facet in $\frH_{n-1}$ along with maps among these at codimension 2 faces, satisfying some conditions.

Let $\alpha$ be a vertex in $\frH_{n-1},$ which is contained in $n+1$ cyclically ordered permutohedra $P_0,\ldots,P_n.$ Then an object $\calF$ in $\calQ_{n-1}^{inf}(\alpha),$ understood as a category of perfect modules over $\langle s_0,\ldots,s_n\rangle,$ is a collection of $n+1$ objects $\calF(s_i)$ and $n+1$ degree 1 maps $\calF(s_i)\leftarrow\calF(s_{i+1})$. 
%
\begin{lemma}
Let $F$ be the facet separating the permutohedra $P_i$ and $P_j,$ and let $\calF|_F$ be the restriction of $\calF$ to $\calQ_{n-1}^{inf}(F)$ (along the inclusion $\alpha\to F$). Then the perfect complex $\calF|_F\in \Fun^{ex}(\langle s_i,s_j\rangle^{op},(\Perf_k)_{\BZ/2})$ is given by
$$
\xymatrix{
\calF|_F(s_i)=(\calF(s_i)&\ar[l]\cdots&\ar[l]\calF(s_{j-1}))&
\calF|_F(s_j)=(\calF(s_j)&\ar[l]\cdots&\ar[l]\calF(s_{i-1})).
}
$$
\end{lemma}
\begin{proof}This follows directly from the definition of the functor $\calQ^{inf}_{n-1}.$
\end{proof}

Let $F_i$ be the facet containing $\alpha$ which separates the permutohedra $P_i$ and $P_{i-1}.$ Then from the above lemma we understand that the object of $(\Perf_k)_{\BZ/2}$ placed at $F_i$ is just $\calF(s_i),$ and at the codimension 2 intersection of $F_i$ and $F_{i+1}$ is the map $\calF(s_i)\leftarrow\calF(s_{i+1}).$

There is also a geometric way of understanding $\cQ_{n-1}^{inf}(\alpha)$ as a category of twisted complexes, for a face $\alpha$ in $\frH_{n-1}$: each of the distinguished generators $s_i$ of the category $\cQ_{n-1}^{inf}(\alpha),$ which are indexed by permutohedra $P_i$ containing $\alpha,$ corresponds to a brane in the Fukaya category which locally near $\alpha$ runs along the interior of the permutohedron $P_i$; the complexes $s_{i,j}$ correspond to branes which cross over to different permutohedra at $\alpha$.

The compatibility conditions mentioned in the approach using perfect modules correspond in this perspective to a list of the possible configurations which a brane can take locally at each face $\alpha.$ If $\alpha$ is a codimension $r$ face, so that it is contained in $r+1$ permutohedra $P_0,\ldots,P_r$, then there are $\frac{r(r+1)}{2}$ possible such configurations, corresponding to the objects $s_i$ and $s_{i,j}$ in the category $\calQ^{inf}(\alpha).$

\begin{example}[$n=2$]\label{ex:braneconfigs}
	Let $\alpha$ be a vertex in $\frH_1,$ which is shared by three hexagons $P_0,P_1,P_2.$ Then locally at $\alpha,$ there are three possible brane configurations $s_0,s_1,s_2$ (up to a shift, these are equivalent to $s_{1,2},s_{2,0},$ and $s_{0,1},$ respectively). These are illustrated in Figure~\ref{figure:braneconfigs}.
\begin{figure}[h]
\includegraphics[width=12cm]{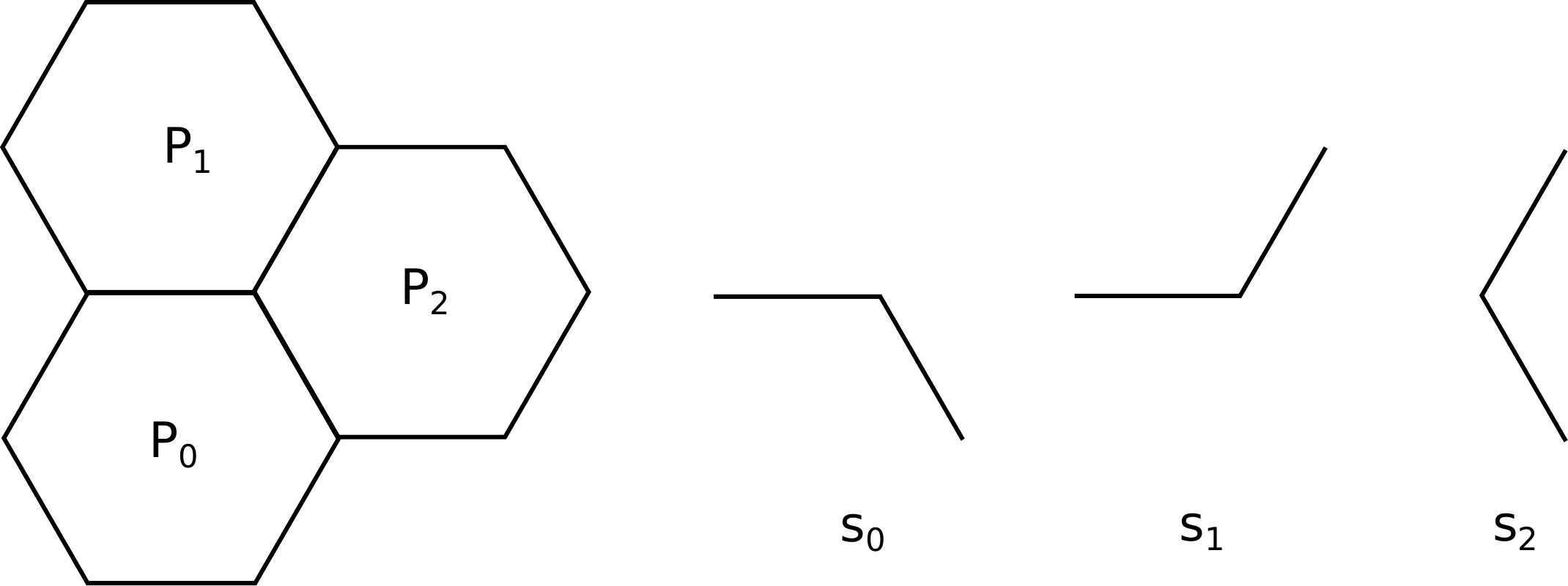}
\caption{A vertex $v$ in $\frH_1$ and the brane configurations corresponding to the generating objects $s_0,s_1,s_2$ of $\cQ^{inf}_{n-1}(v).$}
\label{figure:braneconfigs}
\end{figure}
\end{example}

One basic class of objects in the category $Q^{inf}_{n-1}$ are the ``microlocal rank-one" objects: these are the objects $\calF$ such that for every facet $\alpha$, the object $\calF_\alpha$ is equivalent in $\calQ^{inf}(\alpha)$ to either $s_0$ or $s_1$ (where these correspond to the two permutohedra separated by $\alpha$; cf. Remark~\ref{remark:Sindex}). The microlocal rank-one objects in the Fukaya category perspective are those objects which run along each facet in $\frH_{n-1}$ at most once.

\begin{defn}Let $\calP=\{P_i\}_{i\in I}$ be a set of permutohedra involved in the tiling of $V_{n}.$ Then the boundary $B=\partial (\bigcup_i P_i)$ is a subset of $\frH_{n-1}$ which is a union of strata. (We will occasionally denote $B$ by $\partial \calP.$) A \emph{rank-one brane along $\partial \calP$} is an object of $Q_{n-1}^{inf}$ whose support along each facet $\alpha$ in $B$ is equivalent to $s_0$, where $s_0$ is the generator of $\calQ_{n-1}^{inf}(\alpha)$ corresponding to the permutohedron in $\calP$ containing $\alpha$, and whose support along each facet $\alpha$ not in $B$ is zero. If such an object exists, it is necessarily unique, and we will denote it by $\calB_\calP.$
\end{defn}

\begin{example}\label{example:shersphere}
Let $\calP=\{P_0\}$ be a single permutohedron $P_0.$ Then the object $\calB_\calP$ exists: for any face $\alpha$ in $\frH_{n-1},$ its support $(\calB_\calP)_\alpha\in \cQ^{inf}_{n-1}(\alpha)$ is given by
$$
(\calB_\calP)_\alpha=\begin{cases}s_0& \alpha\in \partial P_0\\ 0&\alpha\notin \partial P_0.\end{cases}
$$
This object, which under the mirror symmetry equivalence presented later in this paper will map to the skyscraper sheaf at the origin of $\AA^{n+1},$ corresponds in the Fukaya category side to the immersed sphere whose endomorphisms were calculated by Sheridan in \cite{Sher1}. 
\end{example}

\begin{example}
Let $n=3$. The 3-permutohedron is the truncated octahedron, which has both hexagon and square facets. Let $\calP=\{P_0,P_0+\lambda_1+\lambda_2\}$ be a set of two permutohedra which share a single square facet. Then there does not exist a rank-one brane along $\partial\calP.$
\end{example}

The second example above shows that we need to institute an additional condition on the set $\calP$ in order to guarantee the existence of a rank-one brane along $\partial P.$ One such condition, which will be sufficient for our purposes, is given in the lemma below.

Recall that for every face $\alpha$ of a permutohedron, we have a cyclic set $\tilcalO(\alpha)$ of all permutohedra containing $\alpha.$

\begin{lemma}\label{lemma:mrank1}
For $\calP$ a set of permutohedra, write $\tilcalO_\calP(\alpha)\subset \tilcalO(\alpha)$ for the the subset of permutohedra in $\tilcalO(\alpha)$ which are contained in $\calP$.
If the subset $\tilcalO_\calP(\alpha)\subset\tilcalO(\alpha)$ is connected in the cyclic order on $\tilcalO(\alpha)$ for every face $\alpha$ in $\partial\calP,$ then the rank-one brane $\calB_\calP$ along $\partial\calP$ exists.
\end{lemma}
\begin{proof} We can define the object $\calB_\calP$ as follows: at any face $\alpha$ not in $\partial\calP$, we set $(\calB_\calP)_\alpha=0.$ At any face $\alpha$ in $\partial\calP$, denote the cyclic set $\tilcalO(\alpha)$ of permutohedra containing $\alpha$ by $[P_0,\ldots,P_r],$ and let $[P_i,P_{i+1},\ldots,P_j]$ denote the cyclic subset $\tilcalO_\calP(\alpha)$ of permutohedra contained in $\calP$. By assumption, this set is connected in the cyclic order on $\tilcalO(\alpha),$ and hence in $\calQ_{n-1}^{inf}(\alpha)$ we can define the complex
$$
\xymatrix{
s_{i,j}:=(s_i\ar[r]&s_{i+1}\ar[r]&\cdots\ar[r]&s_j)
},
$$
and we set $(\calB_\calP)_\alpha=s_{i,j}.$
\end{proof}

We would like to give a similarly explicit description of the category $Q^{wr}_{n-1},$ but the definition above is not well-suited to describing objects of this category, for the reason that colimits of dg categories are more difficult to present than limits are. In order to understand this colimit, we cite from \cite{Ga} the following useful trick, originally due to Jacob Lurie:

\begin{lemma}[{\cite[Lemma 1.3.3]{Ga}}] Let $P$ be a category and $F:P^{op}\to \on{St}^L_k$ a functor to the category of cocomplete $k$-linear dg categories and continuous functors. Let $G:P\to \on{St}_k$ be a functor to the category of cocomplete $k$-linear dg categories which agrees with $F$ on objects and such that $G(\alpha\to \beta)$ is right adjoint to $F(\alpha\to \beta).$ Then there is an equivalence $\on{colim}_{P^{op}}F\cong \lim_{P}G.$
\end{lemma}

By construction, $\calQ_{n-1}^{inf}:P(\frH_{n-1})\to \BZ/2\on{-dgst}_k$ agrees with $\calQ^{wr}_{n-1}$ on objects and $\calQ_{n-1}^{inf}(\alpha\to\beta)$ is right adjoint to $\calQ^{wr}_{n-1}(\alpha\to\beta),$ so we are almost in the situation in the lemma. However, the functors $\calQ_{n-1}^{inf}$ and $\calQ_{n-1}^{wr}$ as defined have codomain all dg categories and not just cocomplete dg categories. We can rectify this by passing to Ind-completions.

Let $\cQ^{\Diamond}$ be the functor defined the same way as the functor $\cQ$, except that its values on objects are equivalent to $(A_n \on{-Mod}_k)_{\BZ/2}$ instead of $(A_n \on{-Perf}_k)_{\BZ/2};$ that is, we allow complexes of any dimension, with no restriction to perfect complexes. Following the procedure by which we defined $Q^{inf}_{n-1},$ we produce in the same way a category $Q_{n-1}^\Diamond,$ which is similar to $Q^{inf}_{n-1}$ but allows infinite-rank stalks along facets. We are now in a position to apply the above lemma.

\begin{cor} The category $\cQ^{wr}_{n-1}$ is equivalent to the category of compact objects in $Q_{n-1}^\Diamond.$
\end{cor}
\begin{proof}The functors $\on{Ind}\circ \cQ^{wr}_{n-1}$ and $\on{Ind}\circ\cQ^{inf}_{n-1}=\cQ^{\Diamond}_{n-1}$ satisfy the conditions of the above lemma, which thus provides an equivalence $\colim_{P(\frH_{n-1})^{op}}(\on{Ind}\circ\cQ_{n-1}^{wr})\cong Q^\Diamond_{n-1}$ between their respective colimit and limit. Passing to the full subcategory of compact objects on each side turns this into an equivalence
$$\xymatrix{
\left(\colim_{P(\frH_{n-1})^{op}}(\on{Ind}\circ\cQ_{n-1}^{wr})\right)^{cpt}\cong (Q^\Diamond_{n-1})^{cpt}.
}$$
Since the Ind-completion commutes with the colimit, the left-hand side of this equivalence is $\left(\on{Ind}(\colim_{P(\frH_{n-1})^{op}}\cQ_{n-1}^{wr})\right)^{cpt}$,
which is just the idempotent-completion of the category $\colim_{P(\frH_{n-1})^{op}}\calQ_{n-1}^{wr}.$
By definition, this latter category is $Q_{n-1}^{wr}.$
\end{proof}

We can use the above lemma to give an explicit description of the generators of $\cQ^{wr}_{n-1}.$ For $F$ a facet in $\frH_{n-1}$ and $\xi$ a choice of normal direction to $F$, consider the map
$$\xymatrix{
\phi_{F,\xi}:Q_{n-1}^\Diamond\ar[r]&\cQ^{\Diamond}_{n-1}(F)\ar[r]^-\xi_-\sim&(\Mod_k)_{\BZ/2}
}
$$
which takes an object of $Q_{n-1}^\Diamond$ to the object of the category $\cQ^{\Diamond}_{n-1}(F)$ (which $\xi$ identifies with $(\Mod_k)_{\BZ/2}$) which is placed at the facet $F$ of $\frH_{n-1}.$ If $\xi_{\pm}$ are the two choices of normal to $F$, then the resulting functors agree up to a shift: $\phi_{F,\xi_+}=\phi_{F,\xi_-}[1].$ (Note that a choice of $\xi$ is equivalent to a choice of one of the two permutohedra containing the facet $F.$ If $P$ is a choice of one of these permutohedra, we will occasionally denote the corresponding functor by $\phi_{F,P}.$)

Since the functor $\phi_{F,\xi}$ preserves products, it admits a left adjoint $\phi_{F,\xi}^\ell:{\Mod_k}\to Q_{n-1}^\Diamond,$ and since it preserves coproducts, $\phi_{F,\xi}^\ell$ preserves compact objects. Hence, if we define $\delta_{F,\xi}$ to be $\phi_{F,\xi}^\ell(k),$ then $\delta_{F,\xi}$ is an object of $Q^{wr}_{n-1}$ which by construction corepresents the functor $\phi_{F,\xi}.$
\begin{defn} For $F$ a facet in $\frH_{n-1},$ the functor $\phi_{F,\xi}$ is a \emph{stalk functor along $F$}. The object $\delta_{F,\xi}$ in $Q^{wr}_{n-1}$ corepresenting $\phi_{F,\xi}$ is a \emph{skyscraper along $F$}.
\end{defn}

It will be useful to restrict our attention to the maximal facets in the honeycomb $\frH_{n-1}$. Recall that the facets of the permutohedron $\calP_n$ are of the form $\calP_{n-k}\times\calP_k$ for $k=1,\ldots,n,$ and that we call ``maximal facets" the facets of the form $\calP_{n-1}.$ Equivalently, these are the facets which are shared by a pair of permutohedra $P$ and $P+\lambda_i$ for some $i$. In the description of $Q_{n-1}^{inf}$ using perfect modules, if an object $\calF_v:\langle s_0,\ldots,s_n\rangle^{op}\to(\Perf_k)_{\BZ/2}$ is placed at a vertex $v$, then $\calF(s_i)$ are its stalks along the $n+1$ maximal facets containing $v.$ We can use this to establish the following lemma:

\begin{lemma} The category $Q^{wr}_{n-1}$ is generated by the set $\{\delta_{F,\xi}\}$ of skyscrapers along maximal facets of $\frH_{n-1}.$
\end{lemma}
\begin{proof}
An object of $Q_{n-1}^\Diamond$ is zero if and only if its stalks along all maximal facets $F$ are zero, which shows that the set of skyscrapers generates $Q_{n-1}^\Diamond.$ Since the category $Q_{n-1}^\Diamond$ is the Ind-completion of its compact objects $Q^{wr}_{n-1}$, the skyscrapers generate $Q^{wr}_{n-1}$.
\end{proof}
%

\begin{remark}For a non-simply-connected symplectic manifold, the category of wrapped microlocal sheaves (modeled here by $Q^{wr}$) lacks the necessary finiteness conditions to embed into the category of infinitesimally wrapped microlocal sheaves (modeled by $Q^{inf}$); instead, both are contained inside a larger category $Q^\Diamond.$ However, passing from $\frH_{n-1}/\Lambda_n$ to its covering space $\frH_{n-1}$ unwraps branes: consider for instance the toy case $\RR\to S^1,$ in which a brane wrapping $S^1$ countably many times might run only once along the universal cover. We might thus expect that the objects in $Q^\Diamond$ which corepresent stalk functors have sufficient finiteness to live inside $Q^{inf}_{n-1}.$ This turns out to be the case, as we will see below.
\end{remark}

The following collection of objects of $Q_{n-1}^{\inf}$ will play an important r\^ole in the proof of the main mirror symmetry equivalence of this paper.

\begin{defn}Let $P$ be a permutohedron in $V_n$, and let $J\subsetneq\{1,\ldots,n+1\}$ be a proper subset. Set $\calP=\{P+\sum_{j\in J} n_j\lambda_j\mid n_j\in \NN\}.$ The hypothesis of Lemma \ref{lemma:mrank1} is satisfied, so this choice of $\calP$ defines a rank-one brane $\calB_\calP,$ which we will denote by $\calB_{P,J}.$
\end{defn}

\begin{figure}[h]
\includegraphics[width=8cm]{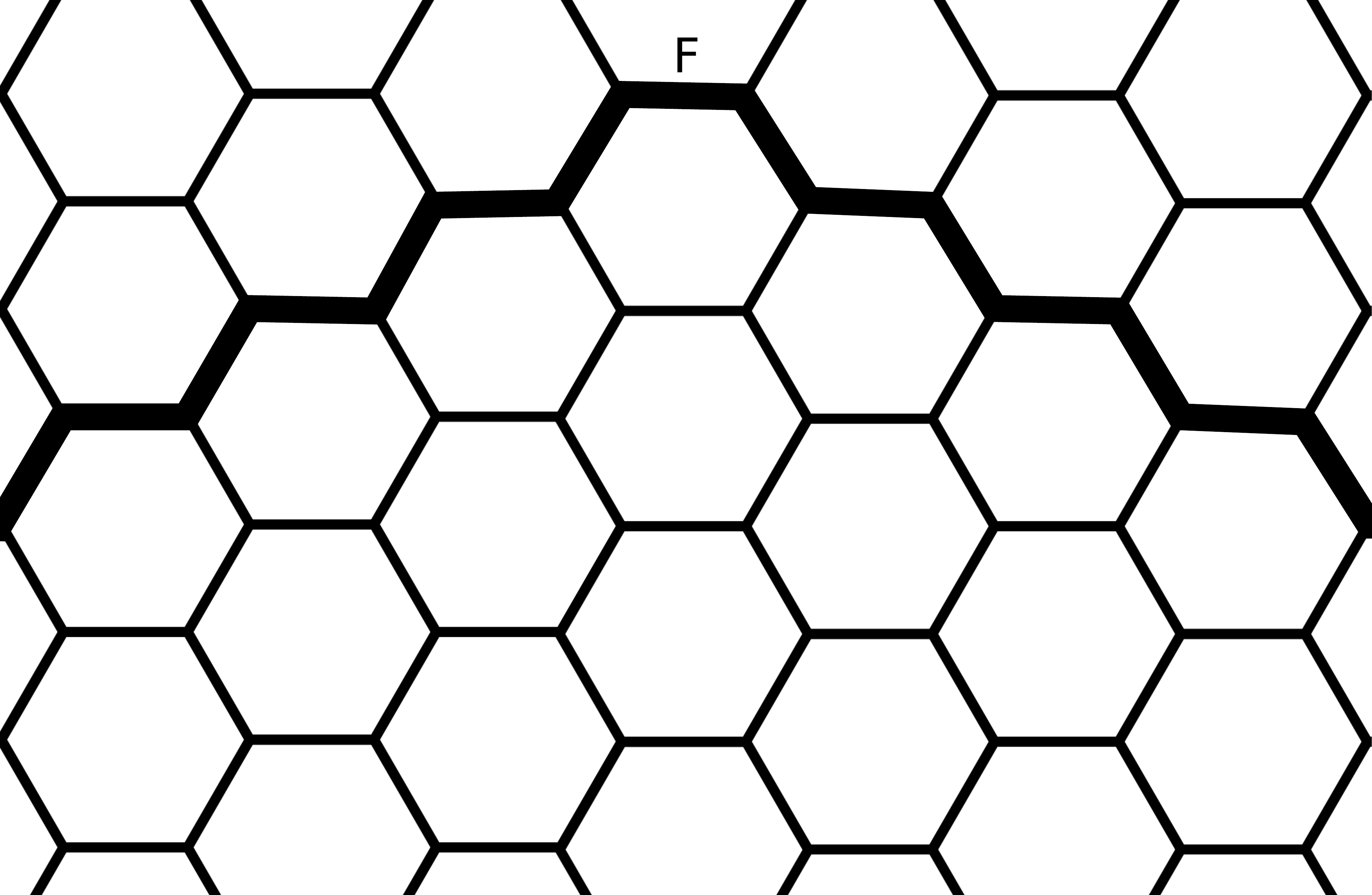}
\caption{The support (in bold) of the skyscraper $\delta_{F,P}=\calB_{P,\{2,3\}}$ along the edge $F.$}
\label{fig:sky}
\end{figure}

The object $\calB_{P,\emptyset}$ is the brane wrapping a single permutohedron, discussed in Example~\ref{example:shersphere} above. At the opposite extreme, in the case where $J=\{i\}^c=\{1,\ldots,n+1\}\setminus\{i\},$ we get a skyscraper:

\begin{prop}\label{prop:skyscrapers}
Let $P$ be a permutohedron in $V_n$ and let $F_i$ be the facet separating $P$ from $P+\lambda_i.$ Then the rank-one object $\calB_{P,\{i\}^c}$ defined above is the skyscraper $\delta_{F_i,P}.$
\end{prop}

\begin{proof}
Let $\calP$ be as in the definition of $\calB_{P,\{i\}^c}.$ As usual we will denote by $B$ the boundary $B=\partial(\bigcup_{P\in\calP}P).$
To show that the object $\calB_\calP$ is isomorphic to the skyscraper $\delta_{F_i,P},$ we need to define an equivalence
$$
\xymatrix{
\Hom_{Q_{n-1}^{inf}}(\calB_\calP,\calG)\ar[r]^-\sim &\Hom_{Q_{n-1}^{inf}}(\delta_{F_i,P},\calG)\cong\phi_{F_i,P}(\calG)
}
$$
which is natural in $\calG.$

If we define a functor
$$
\xymatrix{
h:P(\frH_{n-1})\ar[r]&(\Perf_k)_{\ZZ/2},&
\alpha\ar@{|->}[r]^-h&\Hom_{\calQ(\alpha)}((\calB_\calP)_\alpha,\calG_\alpha),
}
$$
then we have an equivalence
$$
\xymatrix{
\Hom_{Q_{n-1}^{inf}}(\calB_\calP,\calG)\cong\lim_{P(\frH_{n-1})}h,
}
$$
so it would be enough to show that this limit is naturally equivalent to 
$\phi_{F_i,P}(\calG).$ We will calculate this limit by making a series of simplifications until we arrive at the desired result. Heuristically, we will see that after restricting to the support of $B$, the calculation we want can be understood in a category of representations of a certain acyclic quiver.

First, let $P(B)\subset P(\frH_{n-1})$ be the full subposet on faces contained in $B.$ Note that $h(\beta)=0$ for any $\beta\notin P(B),$ and there are no maps $\beta\to\alpha$ for $\alpha\in P(B),\beta\notin P(B).$ Hence the natural map
$$
\xymatrix{
\lim_{P(\frH_{n-1})}h\ar[r]&\lim_{P(B)}h,
}
$$
is an equivalence.

But this latter limit is just the Hom space
$$
\xymatrix{
\lim_{P(B)}h\cong\Hom_{\lim_{P(B)}\calQ_{n-1}^{inf}}(\calB_\calP|_{P(B)},\calG|_{P(B)})
}
$$
of the objects $\calB_\calP$ and $\calG$ after restriction to the category $\lim_{P(B)}\calQ_{n-1}^{inf},$ so we need to compute this latter category, which is equivalent to the category of modules over a certain quiver with relations $(Q_B,R_B)$.

Let $Q_B$ be the quiver with one vertex for every facet in $B$ and one arrow for every codimension 2 face in $B$, with the direction of the arrows determined as in the description of $Q_{n-1}^{inf}$ by perfect modules. For $k>2,$ the two paths around each codimension $k$ face form a non-oriented $k$-cycle, and we add to $R_B$ the relation that these two paths commute.

Then $\lim_{P(B)}\calQ_{n-1}^{inf}$ is equivalent to the category of modules over the quiver $Q_B$ with relations $R_B$.
The quiver representation corresponding to an object $\calF$ in $\lim_{P(B)}\calQ_{n-1}^{inf}$ has at the vertex of $Q_B$ corresponding to the facet $\alpha$ the perfect complex $\phi_{\alpha,\xi}(\calF),$ where $\xi$ is the normal direction along $\alpha$ which points into $\calP$ and $\phi_{\alpha,\xi}$ is the corresponding stalk functor.

The quiver $Q_B$ is a connected quiver with an initial vertex, corresponding to the facet $F_i$; it has no oriented cycles, and any two paths with the same start and endpoint are forced by a relation in $R_B$ to agree. By construction, the object $\calB_\calP$ is mapped by this equivalence to the quiver representation with $k$ placed at every vertex, and every map an isomorphism. This object corepresents the functional on $(Q_B,R_B)$-mod sending a quiver representation to the object of $(\Perf_k)_{\ZZ/2}$ placed at the initial vertex. Hence we have an isomorphism
$$
\xymatrix{
\Hom_{(Q_B\on{-Perf}_k)_{\ZZ/2}}(\calB_\calP|_{P(B)},\calG|_{P(B)})\cong \phi_{F_i,P}(\calG).
}
$$

Composing all of the above equivalences, we conclude that the object $\calB_{\calP}$ corepresents the stalk functor $\phi_{F_i,P},$ as claimed.
\end{proof}

We will also need the following fact about the objects $\calB_{P,J}$, which expresses how they can be built out of one another:
\begin{lemma}\label{lemma:colimrels}Let $P$ be a permutohedron in $V_n,$ let $J'\subset \{1,\ldots,n+1\}$ be a proper subset, and suppose that $J'=J\setminus\{i\}$ for some $i$ and some proper subset $J\subset\{1,\ldots,n+1\}$. Then in $Q^{inf}_{n-1}$ there exists a map $x_i:\calB_{P,J}\to\calB_{P,J}\langle\lambda_i\rangle$ and an isomorphism of complexes
$$\xymatrix{
\calB_{P,J'}\cong(\calB_{P,J}\langle-\lambda_i\rangle\ar[r]^-{x_i}&\calB_{P,J}).
}$$
\end{lemma}

\begin{proof}The definition of the map $x_i$ is clear in the case where $J=\{1,\ldots,n+1\}\setminus\{j\}$ for some $j$. In this case we want to exhibit a map $x_i:\delta_{F_j,P}\langle-\lambda_i\rangle\to\delta_{F_j,P}.$ But since the domain corepresents the stalk functor along the facet $F_j+\lambda_i$ and we know that the codomain is rank-one along this facet, there is a one-dimensional space of maps between these two, so such an $x_i$ exists. Moreover, from the definition of these two objects as the rank-one branes $\calB_{P+\lambda_i,J}$ and $\calB_{P,J},$ we see that the cone on this map is the rank-one brane $\calB_{P,J\setminus\{i\}}.$

To produce the maps on the objects $\calB_{P,J}$ for other $J$, we just need to note that the maps $x_i$ commute, \emph{i.e.}, that the squares
$$
\xymatrix{
\delta_{F_j,P}\langle-\lambda_i-\lambda_k\rangle\ar[r]^-{x_k}\ar[d]_-{x_i}&\delta_{F_j,P}\langle-\lambda_i\rangle\ar[d]^-{x_i}\\
\delta_{F_j,P}\langle-\lambda_k\rangle\ar[r]_-{x_k}&\delta_{F_i,P}
}
$$
are commutative. Hence we can read this square as a map of complexes in two different ways: either vertically, as the map
$$\xymatrix{
x_i:\calB_{P,\{j,k\}^c}\langle-\lambda_i\rangle\to\calB_{P,\{j,k\}^c}
}
$$
or horizontally, as the map
$$\xymatrix{
x_k:\calB_{P,\{j,i\}^c}\langle-\lambda_k\rangle\to\calB_{P,\{j,i\}^c}.
}
$$
We can produce all the maps $x_i$ by iterating this procedure, and they manifestly satisfy the relations described in the lemma. 
\end{proof}

\section{Mirror symmetry}


\subsection{Landau-Ginzburg $B$-model}

Here we recall from \cite{Nwms} the structure of the $B$-brane category associated to the Landau-Ginzburg $B$-model with background $\AA^{n+1}$ and superpotential $W_{n+1}= z_1 \cdots z_{n+1}.$


\subsubsection{Matrix factorizations}

Consider the background $M = \Spec A$, with $A= k[z_1, \ldots, z_{n+1}] $, and a superpotential $W\in A$ such that $0\in \AA^1$ is its only possible critical value.

We will denote by $X$ the special fiber $W^{-1}(0) = \Spec (A /(W))$. 

Let $\Perf(X)$ be the dg category of perfect complexes on $X$,
and $\Coh(X)$ the dg category of bounded coherent complexes of sheaves on $X$.

The category of B-branes associated to the LG model $(M,W)$ is the derived category of singularities $\calD_\sing(X)$, which is defined as the 2-periodic dg quotient category
$$D_\sing(X)= \Coh(X)/\Perf(X).$$

Orlov~\cite{orlov} established an equivalence of the derived category of singularities with the $\ZZ/2$-dg category $\MF(M,W)$ of matrix factorizations associated to $(M,W).$ The objects of this category are pairs $(V, d)$ of a $\ZZ/2$-graded  free $A $-module $V$
of finite rank
equipped with an odd endomorphism $d$
such that $d^2 =W\id$.
Thus we have $V=  V^0 \oplus V^1$, $d= (d_0, d_1) \in \Hom(V^0, V^1)\oplus \Hom(V^1, V^0)$, and
$d^2 =  (d_1 d_0, d_0 d_1) =(W\id, W\id)  \in \Hom(V^0, V^0)\oplus \Hom(V^1, V^1)$.
We denote the data of a matrix factorization by a diagram
$$
\xymatrix{
V^0 \ar[r]^-{d_0} &  V^1 \ar[r]^-{d_1}   & V^0.
}
$$

Orlov's equivalence
$$\xymatrix{
\MF(M, W)_{2\ZZ} \ar[r]^-\sim &\calD_\sing(X)
}
$$
is given by
$$
\xymatrix{ 
(V^0 \ar[r]^-{d_0} &  V^1 \ar[r]^-{d_1}   & V^0) \ar@{|->}[r] & \coker(d_1).
}
$$


\subsubsection{Coordinate hyperplanes}

For $n\in \mathbb N$, set  $[n] = \{1, \ldots, n\}$.

In this paper, we are interested in the matrix factorization category associated to the background $\AA^{n+1} = \Spec A_{n+1}$, with $A_{n+1} = k[z_a \, |\, a\in [n+1]] $, and the superpotential 
  $$
  \xymatrix{
  W_{n+1} = z_1 \cdots z_{n+1} \in A_{n+1}.
  }
  $$
This LG model is mirror to the pair of pants.

The special fiber of this superpotential is
  $$
  \xymatrix{
  X_n= W_{n+1}^{-1}(0) = \Spec B_n,
  }$$ 
  where we set $B_n =A_{n+1} /(W_{n+1})$.

It will also be convenient to set $W_{n+1}^a = W_{n+1}/z_a \in A$, for $a\in [n+1]$.

For $a\in [n+1]$, let $X_{n}^a =  \Spec A/(z_a) \subset X_n$ denote the coordinate hyperplane,  
and  $\cO_n^a$  its structure sheaf. As an object of $\Perf(\AA^{n+1})$, it admits the free resolution
$$
\xymatrix{
A_{n+1} \ar[r]^-{z_a} & A_{n+1} \ar[r] & \cO_n^a,
}
$$
and as an object of $\Coh(X_{n})$, it admits the infinite resolution
$$
\xymatrix{
\cdots \ar[r]^-{W_{n+1}^a} & B_n \ar[r]^-{z_a} & B_n \ar[r]^-{W_{n+1}^a} & B_n \ar[r]^-{z_a} & B_n \ar[r] & \cO_n^a.
}
$$

For $a\in [n+1]$,  let $\ul \cO_{n}^a\in \MF(\AA^{n+1}, W_{n+1})$ denote the matrix factorization
$$
\xymatrix{
A_{n+1} \ar[r]^-{W_{n+1}^a} &  A_{n+1}  \ar[r]^-{z_a}   & A_{n+1}.
}
$$

\begin{prop}\label{prop:mfhoms}
The $\ZZ/2$-dg category $\MF(\AA^{n+1}, W_{n+1})$ is split-generated by the collection of objects
$ \ul \cO_n^a$, for $a\in [n]$. There are equivalences of $\ZZ/2$-graded $k$-modules
$$
\xymatrix{
H^*(\Hom(\ul \cO_n^a, \ul \cO_n^a)) \simeq  A_{n+1}/(z_a, W_{n+1}^a),
&
a\in [n+1]
}
$$
$$
\xymatrix{
H^*(\Hom(\ul \cO_n^a, \ul \cO_n^b)) =
 A_{n+1}/(z_a, z_b)[-1],
&
a\not = b\in [n+1].
}
$$
\end{prop}

\begin{proof}
 The collection of objects $ \cO_n^a$, for $a\in [n+1]$, generates $\Coh(X_n)$,    
and $\ul \cO_n^{n+1}$ is in the triangulated envelope of the collection of objects $\ul \cO_n^a$, for $a\in [n]$,
hence
the collection of objects $\ul \cO_n^a$, for $a\in [n]$, generates $\MF(\AA^{n+1}, W_{n+1})$.
The cohomology of morphism complexes is a straightforward calculation.
\end{proof}


\subsection{Equivariant/Graded version}\label{s:equiv}

In order to match the passage to a universal abelian cover on the pair of pants, we must pass to a quotient on the mirror. Equivalently, we must work with a $B$-model category which has been enhanced by equivariance data.

Let us return to the general setup $M = \Spec A$, with 
 $A= k[z_1, \ldots, z_n]$,
 and now assume the superpotential $W\in A$ is homogeneous for an algebraic torus $T\subset (\G_m)^{n}$.

Let  $\chi^*(T) =\Hom(T, \G_m)$ denote the weight lattice of $T$,
and   $w\in \chi^*(T)$ the weight of $W$.

Recall that a $T$-equivariant $k$-module is equivalently a $T$-representation, or equivalently again a $\chi^*(T)$-graded $k$-module. Given a $T$-equivariant $k$-module $V$, 
 we write $V_\lambda$ for the $\lambda$-component of $V$,
 for  $\lambda\in \chi^*(T)$.
 Given a $T$-equivariant $k$-module $V$, and a weight $\mu\in \chi^*(T)$,
we have the $\mu$-twisted $k$-module 
defined by $
V\langle\mu\rangle_\lambda = V_{\lambda - \mu},
$
 for $\lambda\in \chi^*(T)$.

As before, let $X = W^{-1}(0) = \Spec B$ denote the special fiber, with $B  = A /(W)$.

Let $\Perf(X)^T$ be the dg category of $T$-equivariant perfect complexes on $X$,
and $\Coh(X)^T$ the dg category of  $T$-equivariant bounded coherent complexes of sheaves on $X$.

Let $\D_\sing(X)^T = \Coh(X)^T/\Perf(X)^T$ be the dg quotient category of 
 $T$-equivariant singularities. Note that $\D_\sing(X)^T$ is not a 2-periodic dg category, but rather the shift $[2]$ is equivalent to the twist 
 $\langle w\rangle$.

Let $\MF(M, W)^T$ be the  dg category of $T$-equivariant matrix factorizations.
Its objects are pairs $(V, d)$ of a $\BZ/2$-graded  free $T$-equivariant $A$-module 
$V=  V^0 \oplus V^1$
of finite rank
together with a $T$-equivariant morphism
$d= (d_0, d_1) \in \Hom(V^0\langle-w\rangle, V^1)\oplus \Hom(V^1, V^0)$ such that
$d^2 =  (d_1 d_0, d_0 d_1) =(W\id, W\id)  \in \Hom(V^0\langle-w\rangle, V^0)\oplus \Hom(V^1\langle-w\rangle, V^1)$.
We denote the data of a matrix factorization by a diagram
$$
\xymatrix{
V^0\langle-w\rangle \ar[r]^-{d_0} &  V^1 \ar[r]^-{d_1}   & V^0,
}
$$
or equivalently by its $w$-twisted periodicization
$$
\xymatrix{
\cdots \ar[r] &  V^1\langle-w\rangle \ar[r]^-{d_1}   & V^0\langle-w\rangle \ar[r]^-{d_0} &  V^1 \ar[r]^-{d_1}   & V^0\ar[r]^-{d_0}  & V^1\langle w\rangle \ar[r] & \cdots.
}
$$
The morphism complex between $T$-equivariant matrix factorizations is the usual  $T$-equivariant morphism complex between their $w$-twisted periodicizations. (Hence if $T$ is the trivial torus, then $\MF(M, W)^T$ is the usual $\BZ/2$-dg category $\MF(M, W)$ of plain matrix factorizations, considered as a 2-periodic dg category.)
Note that  the shift $[2]$  is equivalent to the  twist $\langle w\rangle$.

As in the non-equivariant case, there is an equivalence of dg categories
$$\xymatrix{
\MF(M, W)^T \ar[r]^-\sim &\D_\sing(X)^T
& 
(V = V^0 \oplus V^1, d = (d_0, d_1)) \ar@{|->}[r] & \coker(d_1)
}
$$

Now let us focus on 
the background $\BA^{n+1} = \Spec A_{n+1}$, with $A_{n+1} = k[z_a \, |\, a\in [n+1]] $, and the superpotential 
  $$
  \xymatrix{
  W_{n+1} = z_1 \cdots z_{n+1} \in A_{n+1}.
  }
  $$
  
  Recall the union of coordinate hyperplanes 
  $$
  \xymatrix{
  X_n= W_{n+1}^{-1}(0) = \Spec B_n,
  }$$ 
  where we set $B_n =A_{n+1} /(W)$.

Now consider the entire torus $T_{n+1} = (\G_m)^{n+1}$ with weight lattice $\chi^*(T_{n+1}) \simeq \BZ^{n+1}=\BZ\langle\lambda_1,\ldots,\lambda_{n+1}\rangle$. We will be interested in the subtorus $\TT^n$ which is the kernel of the restriction of the superpotential $W_{n+1}=z_1\cdots z_{n+1}$ to $T^n.$ The torus $\TT^n$ has weight lattice $\Lambda_n=\BZ\langle \lambda_1,\ldots,\lambda_{n+1}\rangle/(\sum \lambda_i).$ (As before we use $\lambda_i$ to denote the class of $\lambda_i$ in $\chi^*(\TT^n).$ Since we will never be interested in the torus $T^{n+1},$ this ambiguity poses no problems for us.)
As a subtorus of $T_{n+1},$ the torus $\TT^n$ inherits a natural action on $\BA^{n+1},$ equipping $A_{n+1}$ with the $\Lambda_n$-grading for which the coordinate function $z_a \in A_{n+1}$
has weight $\lambda_a$, for $a \in [n+1]$.

We will be interested in the $\TT^n$-equivariant matrix factorization category $\MF(\BA^{n+1},W_{n+1})^{\TT^n}.$ Note that since the superpotential $W_{n+1} \in A_{n+1}$ has weight 0 for the $\TT^n$ action, the shift $[2]$ in this category is actually equivalent to the identity, and hence $\TT^n$-equivariant matrix factorizations actually form a $\BZ/2$-dg category.

For $a\in [n+1]$,  let $\ul \cO_{n}^a\in \MF(\BA^{n+1}, W_{n+1})^{\TT^n}$ denote the $\TT^n$-equivariant matrix factorization
$$
\xymatrix{
A_{n+1} \ar[r]^-{W_{n+1}^a} &  A_{n+1}\langle\lambda_a\rangle  \ar[r]^-{z_a}   & A_{n+1}.
}
$$

We have the following elaboration of Proposition~\ref{prop:mfhoms}.

 \begin{prop}
 The dg category $\MF(\BA^{n+1}, W_{n+1})^{\TT^n}$ is generated by the collection of objects
$ \ul \cO_n^a\langle\lambda\rangle$, for $a\in [n]$, and $\lambda\in\Lambda_n$. There are equivalences of $\BZ/2$-graded $k$-modules
$$
\xymatrix{
H^*(\Hom(\ul \cO_n^a\langle \lambda\rangle, \ul \cO_n^a\langle \mu\rangle)) \simeq  A_{n+1}/(z_a, W_{n+1}^a)_{\lambda - \mu},
&
a\in [n+1]
}
$$
$$
\xymatrix{
H^*(\Hom(\ul \cO_n^a\langle \lambda\rangle, \ul \cO_n^b\langle \mu\rangle)) \simeq
 A_{n+1}/(z_a, z_b)[-1]_{\lambda - \mu+\lambda_a},
&
a\not = b\in [n+1].
}
$$
 \end{prop}
\begin{proof}
This is the same calculation as in Proposition~\ref{prop:mfhoms} but restricted to the subcomplex of $\TT^n$-equivariant maps. The extra twist by $\lambda_a$ in the second equality is a result of our choice to twist the degree 1 (rather than degree 0) piece when defining the equivariant complex $\ul\cO_n^a.$
\end{proof}

We highlight also one additional piece of structure which is useful for understanding the equivalence proved in the next subsection.
\begin{defn}
We will let 
$$
\xymatrix{
f_{ij}:\ul\cO_n^i[1]\ar[r]&\ul\cO_n^j\langle \lambda_i \rangle
}
$$
be the (closed, degree 0) map of matrix factorizations which is given by
$$
\xymatrix{
	A_{n+1}\langle\lambda_i\rangle\ar[r]^-{-z_i}\ar[d]^{-id} & A_{n+1}\ar[r]^-{-W^i_{n+1}}\ar[d]^{W_{n+1}^{i,j}} & A_{n+1}\langle\lambda_i\rangle \ar[d]^{-id } \\ 
A_{n+1}\langle\lambda_i\rangle\ar[r]^-{W_{n+1}^j} & A_{n+1}\langle\lambda_j+\lambda_i\rangle\ar[r]^-{z_j} & A_{n+1}\langle\lambda_i\rangle,
}
$$
where we write $W_{n+1}^{i,j}$ for $\frac{W_{n+1} }{z_iz_j}.$
\end{defn}
The map $f_{ij}$ is a representative for
$$1\in  A_{n+1}/(z_i, z_j)[-1]_0\cong H^*(\Hom_{\MF(\BA^{n+1},W_{n+1})^{\TT^n}}(\ul \cO_n^i, \ul \cO_n^j\langle \lambda_i\rangle)),$$
and the collection of maps $f_{ij}$ (together with their twists by $\lambda\in \Lambda$) form a set of generating morphisms for the category $\MF(\BA^{n+1},W_{n+1})^{\TT^n}.$

\begin{lemma}\label{lemma:acyclic}
	Let $I=\{i_1,\ldots,i_k\}\subset \{1,\ldots,n+1\}$ be a nonempty subset equipped with an ordering. By taking successive extensions of the $\ul\cO_n^{i_j}$ along the morphisms $f_{i_j,i_{j+1}}$, we can define a twisted complex
$$
\ul\cO_n^I:=\left(\xymatrix{
		\ul\cO_n^{i_1}\ar[r]^-{f_{i_1i_2}}&\ul\cO_n^{i_2}\langle\lambda_{i_1}\rangle\ar[r]^-{f_{i_2i_3}}&\cdots\ar[r]^-{f_{i_{k-1}i_{k}}}&\ul\cO_n^{i_{k}}\langle\sum_{j=1}^k\lambda_{i_{j}}\rangle
}\right)
$$
(where we leave implicit in our notation the homotopies witnessing the triviality of compositions) which is independent of the choice of ordering on $I$, up to a shift of $\Lambda$-grading. Moreover, if $I=\{1,\ldots,n+1\}$, then $\ul\cO_n^I=0.$
\end{lemma}
\begin{proof}
	In order to simplify notation, we will work nonequivariantly (\emph{i.e.}, forgetting the $\Lambda$-grading).
	Now the objects $\ul\cO_n^I$ become easy to understand if we work in the derived category of singularities 
	$D_{sing}(\Spec (A_{n+1}/W_{n+1}))$ instead of the matrix factorization category.
	The equivalence between these two categories takes the matrix factorization $\ul\cO_n^i$ to the structure sheaf $\cO_n^i=A_{n+1}/(z_i)$ of the hyperplane $\{z_i=0\},$ and it takes the degree 1 map $f_{ij}$ to the extension
$$
\xymatrix{
A_{n+1}/(z_j)\ar[r]^-{z_i}&A_{n+1}/(z_iz_j)\ar[r]^-{z_j}&A_{n+1}/(z_i).
}
$$

Similarly, the map $f_{jk}:\ul\cO_n^j[1]\to \ul\cO_n^k$ descends to a map
$\ul \cO_n^{\{i,j\}}[1]\to \ul\cO_n^k$ which in the singularity category is equivalent to the extension
\[\xymatrix{
A_{n+1}/(z_k)\ar[r]^-{z_iz_j}&A_{n+1}/(z_iz_jz_k)\ar[r]^-{z_k}&A_{n+1}/(z_iz_j).
	}
\]
%
%

By iterating this process, we see that $\ul\cO^I_n$ is represented in
$D^b_{sing}(\Spec(A_{n+1}/W_{n+1}))$ by $A_{n+1}/(z_{i_1}\cdots z_{i_k}),$ the
structure sheaf of the union of the hyperplanes $\{z_{i_j}=0\}$. This
proves the final statement of the lemma, since if $I=\{1,\ldots,n+1\},$ then
$\ul\cO^I_n$ is represented by a free rank-1 module over $A_{n+1}/W_n,$ which
is zero in the singularity category.
\end{proof}


\subsection{Main result}

The main result of this paper will be a $\tilW_n$-equivariant equivalence between the equivariant matrix factorization category $\MF(\AA^{n+1},W_{n+1})^{\TT^n}$
and the combinatorial Fukaya category $Q^{wr}_{n-1}$ constructed in the previous section. We will establish this equivalence by describing a functor
$$\xymatrix{
\Coh(W_{n+1}^{-1}(0))^{\TT^n}\ar[r]^-{\bar\Phi}&Q^{inf}_{n-1}
}$$
and checking that it factors through both the projection
$$
\xymatrix{\Coh(W_{n+1}^{-1}(0))^{\TT^n}\ar@{->>}[r]&\Coh(W_{n+1}^{-1}(0))^{\TT^n}/\Perf(W_{n+1}^{-1}(0))^{\TT^n} \cong\MF(\AA^{n+1},W_{n+1})^{\TT^n}
}$$
and the inclusion
$$
\xymatrix{
Q^{wr}_{n-1}\ar@{^{(}->}[r]&Q^{inf}_{n-1},
}
$$
and that the middle functor $\Phi$ in the resulting sequence of functors
$$
\xymatrix{
\Coh(W_{n+1}^{-1}(0))^{\TT^n}\ar@{->>}[r]&\MF(\AA^{n+1},W_{n+1})^{\TT^n}\ar[r]^-{\Phi}&Q^{wr}_{n-1}\ar@{^{(}->}[r]&Q^{inf}_{n-1}
}
$$
is an equivalence of categories.

In order to define a functor with domain $\Coh(W_{n+1}^{-1}(0)),$ we use the fact that the variety $W_{n+1}^{-1}(0)$ can be obtained by gluing together copies of affine space: Let $D$ be the poset of proper subsets $I$ of the set $\{1,\ldots,n+1\},$ and write $\AA^I$ for $\AA^{|I|}.$ Then the natural inclusion maps $\AA^I\to W_{n+1}^{-1}(0)$ and the inclusion maps $\AA^I\to \AA^J$ induced by inclusions $I\subset J$ give a $D$-diagram of varieties, and we have an equivalence
$$
\xymatrix{
\colim_{D}\AA^I\ar[r]^-\sim&W_{n+1}^{-1}(0).
}
$$

This induces an equivalence
$$
\xymatrix{
\colim_D\Coh(\AA^I)^{\TT^n}\ar[r]^-\sim &\Coh(W_{n+1}^{-1}(0))^{\TT^n},
}
$$
so that the functor $\bar\Phi$ will be an object of
$$\Fun(\colim_D\Coh(\AA^I)^{\TT^n},Q_{n-1}^{inf})^{\tilW_n}=\lim_D\Fun(\Coh(\AA^I)^{\TT^n},Q_{n-1}^{inf})^{\tilW_n}.$$

A $\tilW_n$-equivariant functor from $\Coh(\AA^I)^{\TT^n}$ is just a choice of object $\calO_I$ with commuting maps $x_i:\calO_I\langle-\lambda_i\rangle\to \calO_I$ for each $i\in I.$ Since the limit diagram
$$D\ni I\mapsto \Fun(\Coh(\AA^I)^{\TT^n},Q_{n-1}^{inf})^{\tilW_n}$$
is strict, objects of this limit can be defined ``by hand," without any higher coherence data: such an object is a choice of a $\tilW_n$-equivariant functor $\Coh(\AA^I)^{\TT^n}\to Q_{n-1}^{inf}$ for each $I$, plus coherent equivalences
$$
\xymatrix{
\calO_{I\setminus\{i\}}\cong(\calO_I\langle-\lambda_i\rangle\ar[r]^-{x_i}&\calO_I).
}
$$

According to the above analysis, we can define a functor $\bar\Phi$ as follows: fix once and for all a permutohedron $P$ in $V_n.$ Then we define $\bar\Phi$ by declaring $\calO_I=\calB_{P,I}$ and the maps $x_i$ to be the maps from Lemma~\ref{lemma:colimrels}.
\begin{lemma}These choices satisfy the necessary relations to define a functor 
$$\xymatrix{
\Coh(W^{-1}(0))^{\TT^n}\cong \on{colim}_D\Coh(\AA^I)^{\TT^n}\ar[r]^-{\bar\Phi}&Q^{inf}_{n-1}.
}
$$
\end{lemma}
\begin{proof}This is exactly the content of Lemma \ref{lemma:colimrels}.
\end{proof}

Note that this means in particular that for $I=\{1,\ldots,n+1\}\setminus\{i\},$ the structure sheaf $\calO_{\AA^I}$ is mapped by $\bar\Phi$ to the skyscraper $\delta_{F_i,P},$ where $F_i$ is the facet separating $P$ from $P+\lambda_i.$

\begin{theorem}\label{theorem:main}
The functor $\bar\Phi$ can be factored as a composition
$$
\xymatrix{
\Coh(W_{n+1}^{-1}(0))^{\TT^n}\ar[r]&\MF(\AA^{n+1},W_{n+1})^{\TT^n}\ar[r]^-{\Phi}&Q^{wr}_{n-1}\ar[r]&Q^{inf}_{n-1},
}
$$
where the left-hand map is the projection, the right-hand map is the inclusion, and the middle map $\Phi$ is an equivalence of categories equivariant for the $\tilW_n$ action.

\end{theorem}
\begin{proof}
To see that $\bar\Phi$ factors through the projection
$$\xymatrix{
\Coh(W_{n+1}^{-1}(0))^{\TT^n}\ar[r]&\Coh(W_{n+1}^{-1}(0))^{\TT^n}/\Perf(W_{n+1}^{-1}(0))^{\TT^n}\cong
\MF(\AA^{n+1},W_{n+1})^{\TT^n},
}$$
we need only to check that the structure sheaf $\calO_{W_{n+1}^{-1}(0)}$ is sent to 0 by $\bar\Phi.$ The structure sheaf $\calO_{W_{n+1}^{-1}(0)}$ of the colimit $\colim\AA^I$ is presented as the limit of the structure sheaves $i_{I*}\calO_{\AA^I}$ (where $i_{I}$ is the inclusion of $\AA^I$ into $\colim\AA^I$). The image of this object under $\bar\Phi$ is the limit of the rank-one branes $\calB_{P,I},$ which is zero, as required.

Hence $\bar\Phi$ does indeed induce a map $\MF(\AA^{n+1},W_{n+1})^{\TT^n}\to Q_{n-1}^{inf}$. Moreover, by construction this map sends the generators $\ul\cO_n^a$ to the skyscrapers $\delta_{F_i,P},$ which generate $Q_{n-1}^{wr}$, and so we see that $\bar\Phi$ factors through a map
$$
\xymatrix{
\Phi:\MF(\AA^{n+1},W_{n+1})^{\TT^n}\ar[r]& Q_{n-1}^{wr}.
}
$$

To show that this functor $\Phi$ is an equivalence, it suffices to check that each of the generating morphisms
$$
\xymatrix{
\ul\cO_n^i[1]\ar[r]^-{f_{ij}}&\ul\cO_n^j\langle \lambda_i \rangle
}
$$
for the category $\MF(\AA^{n+1},W_{n+1})^{\TT^n}$ is sent by $\Phi$ to the unique nonzero morphism
$$
\xymatrix{
\delta_{F_i,P}[1]\ar[r]&\delta_{F_j,P}\langle\lambda_i\rangle,
}
$$
which we will denote by $g_{ij}.$ This follows from the fact that a representative for $f_{ij}$ in the colimit presentation of $\Coh(W_{n-1}^{-1}(0))^{\TT^n}$ is the map presenting $\calO_n^j$ as the cone on the map
$$
\xymatrix{
\lim(\calO_{\AA^{\{i\}^c}}\ar[r]&\calO_{\AA^{\{i,j\}^c}}&\calO_{\AA^{\{j\}^c}}\ar[l])[1]\ar[r]&\cO_{\AA^{\{i\}^c}}[1],
}
$$
so that $\Phi(f_{ij})$ is the map presenting $\delta_{F_j,P}$ as the cone on
$$
\xymatrix{
	\lim(\calB_{P,\{i\}^c}\ar[r]&\calB_{P,\{i,j\}^c}&\ar[l]\calB_{P,\{j\}^c})[1]\ar[r]&\calB_{P,\{i\}^c}[1]
}.
$$
But this is a presentation of the map $g_{ij},$ as desired. We conclude that $\Phi$ is a $\tilW_n$-equivariant equivalence of categories.
\end{proof}
%
%

\section{Symplectic geometry}

So far in this paper we have described a category $Q_{n-1}^{wr}$ and shown that it is equivalent to $\Coh(\AA_{n+1},W_{n+1})^{\TT^n}.$ However, we have not yet explained why the category $Q_{n-1}^{wr}$ is the $A$-model associated to the $\Lambda_n$-cover $\tilcalP_{n-1}$ of the pair of pants. In this section, we will recall our perspective on the $A$-model of a Weinstein manifold as a category of wrapped microlocal sheaves on a Lagrangian skeleton, and, using the skeleton for the pair of pants described in \cite{Nwms}, we will show that our category $Q_{n-1}^{wr}$ is the $A$-model category associated to $\tilcalP_{n-1}$ in this formalism. This establishes the main equivalence of our paper as an instance of homological mirror symmetry.


\subsection{Microlocal $A$-model}
We recall here some properties of microlocal sheaf categories. We refer to \cite{KS} for definitions and and a full exposition of the theory of microlocal sheaves, and to~\cite{Nwms} for a brief review of the theory along with the definition of the wrapped microlocal sheaf categories.

\label{sec:A-cosheaf}


\subsubsection{Setup}

Let $Z$ be a real-analytic manifold. We will denote by $\Sh^\un(Z)$ the dg category of  all complexes of sheaves of $k$-vector spaces on $Z$
for which there exists a Whitney stratification
$\cS=\{Z_\alpha\}_{\alpha\in A}$ of $Z$
such that for each stratum $Z_\alpha \subset Z$, the total cohomology sheaf of the restriction
$\cF|_{Z_\alpha}$ is locally constant. We will denote by 
$\Sh(Z)$ the full subcategory of $\Sh^\un(Z)$ on the sheaves whose cohomology sheaves on each stratum are finite rank. 

We would like to consider the subcategories of $\Sh^\un(Z)$ defined by singular support conditions, which we recall now. Fix a point $(z, \xi) \in T^* Z$. Let $B\subset Z$ be an open ball around $z\in Z$,
and $f:B\to \BR$ a smooth function such that $f(z) = 0$ and $df|_z = \xi$.  
We will refer to $f$ as a compatible test function.
 
Then the \emph{vanishing cycles functor} $\phi_f$ associated to the function $f$ is defined by
 $$
 \xymatrix{
 \phi_f:\Sh^\un(Z)\ar[r] & \Mod_k,
 }
 $$
 $$
 \xymatrix{
 \phi_f(\cF) = \Gamma_{\{f\geq 0\}}(B, \cF|_B) \simeq  
 \Cone(\Gamma(B, \cF|_B) \to \Gamma(\{f< 0\}, \cF|_{ \{f< 0\}}))[-1],
 }
 $$
 where we take $B \subset Z$ sufficiently small.
 In other words, we take sections of $\cF$ over the ball $B$ supported where $f\geq 0$, or equivalently vanishing where $f< 0$. 

To any object $\cF \in\Sh^\un(Z) $, we can associate its
\emph{singular support}
$$
\xymatrix{
\ssupp(\cF) \subset T^* Z
}$$ 
to be the smallest closed subset such that $\phi_f(\cF) \simeq 0$, for any 
$(z, \xi) \in T^*Z \setminus \ssupp(\cF)$, and any compatible test function $f$. The singular support $\ssupp(\cF)$ is a closed conic Lagrangian subvariety of $T^*Z.$

For a conic Lagrangian subvariety $\Lambda\subset T^*Z$,
we write $\Sh_{\Lambda}^\un(Z) \subset \Sh^\un(Z)$
(respectively $\Sh_{\Lambda}(Z)    \subset \Sh(Z)$) 
 for the full dg subcategory of objects $\cF \in \Sh^\un(Z)$ (respectively $\cF \in \Sh(Z)$)  with singular support satisfying 
$
\ssupp(\cF) \subset \Lambda.
$

\subsubsection{Microlocal sheaf categories}\label{s:microsh}

Now we can recall the definition of the microlocal sheaf and wrapped microlocal sheaf categories associated to a conic Lagrangian.

Let $\Lambda\subset T^*Z$ be a closed conic Lagrangian subvariety. To $\Lambda$ we can associate a conic sheaf of dg categories $\mu\Sh^\un_{\Lambda}$ on $T^*Z$ which is supported on $\Lambda$. Its global sections $\mu\Sh^\un(T^*Z)$ form the category of \emph{large microlocal sheaves along $\Lambda$.}

Since $\mu\Sh^\un_{\Lambda}$ is a sheaf, its definition can be stated locally. Let $(z,\xi)\in T^*Z$, and let $\Omega$ be a small conic open neighborhood of $(z,\xi).$ We will write $B=\pi(\Omega)$ for the projection of $\Omega$ to a small neighborhood of $z$ in $Z.$

If $\xi=0,$ so that $\Omega=T^*B,$ then 
we have a natural equivalence
$$
\xymatrix{
\Sh^\un_\Lambda(B) \ar[r]^-\sim & \mu\Sh^\un_\Lambda(\Omega)
}
$$
of the category of large microlocal sheaves along $\Lambda\cap\Omega$ with the category of large constructible sheaves on $B$ with singular support in $\Lambda.$

If $\xi\neq 0,$ so that $\Omega\cap Z=\emptyset,$ then the category of large microlocal sheaves on $\Lambda\cap\Omega$ is naturally equivalent to a dg quotient category,
 
$$
\xymatrix{
\Sh^\un_{\Lambda}(B, \Omega)/K^\un(B, \Omega)\ar[r]^-\sim  & \mu \Sh^\un_{\Lambda}(\Omega),
}
$$
where $\Sh^\un_{\Lambda}(B, \Omega)\subset \Sh^\un(B)$
is the full dg subcategory of objects $\cF\in \Sh^\un(B)$ with singular support satisfying $\ssupp(\cF) \cap \Omega \subset \Lambda$ and $K^\un(B, \Omega)\subset\Sh^\un_{\Lambda}(B, \Omega)$ denotes the full dg subcategory of 
 objects $\cF\in \Sh^\un(B)$ with singular support satisfying $\ssupp(\cF) \cap \Omega = \emptyset$.

The main fact we will need about the calculation of these microlocal sheaf categories is the calculation, done in \cite{Narb}, that the category of microlocal sheaves on an arboreal singularity of type $A_n$ is equivalent to the category of modules over the $A_n$ quiver.

Now we recall from \cite{Nwms} the category of wrapped microlocal sheaves:
\begin{defn}
The category of \emph{wrapped microlocal sheaves along $\Lambda\cap \Omega$} is the full dg subcategory
$$
\xymatrix{
\mu\Sh^{wr}_\Lambda(\Omega)\subset \mu\Sh^\un_\Lambda(\Omega).
}
$$
of compact objects inside the category $\mu\Sh^\un_\Lambda(\Omega)$ of big microlocal sheaves.
\end{defn}
In that paper was proved the following fact:
\begin{prop}[\cite{Nwms} Proposition 3.16] The categories $\mu Sh^{wr}_\Lambda(\Omega)$ assemble into a cosheaf of categories on $\Lambda$.
\end{prop}
We will refer to the global sections of this cosheaf as the category of \emph{wrapped microlocal sheaves along $\Lambda.$}

\begin{remark}The cosheaf of wrapped microlocal sheaf categories as defined above is a dg rather than $\BZ/2$-dg category--\emph{i.e.}, it possesses a natural $\ZZ$-grading, equivalent to the canonical grading on the Fukaya category of a cotangent bundle. However, later on, we will be interested in gluing together different cotangent bundles, where these gradings will no longer agree (unless we make some additional choices). Thus, we will forget the $\ZZ$-grading on $\mu\Sh^{wr}_\Lambda$ and for the rest of this paper will work instead with a $\BZ/2$-graded version, which we denote by $(\mu\Sh^{wr}_\Lambda)_{\BZ/2}.$
\end{remark}

\subsubsection{Skeleta and quantization categories}
Now we are almost ready to discuss the relation of this paper to Fukaya categories. Recall first the definition of a Weinstein manifold:
\begin{defn}A \emph{Weinstein manifold} $(W,\omega,Z,h)$ is a symplectic manifold $(W,\omega)$ along with a vector field $Z$ satisfying the Liouville condition $\calL_Z\omega=\omega$ and a Morse function $h:W\to\RR$ for which the Liouville field $Z$ is gradient-like.
\end{defn}

We will write $\lambda$ for the Liouville 1-form (corresponding to $Z$ under the equivalence given by $\omega$), and we will often refer to the Weinstein manifold $(W,\omega,Z,h)$ by $W$ when the other data are understood. The basic references for the theory of Weinstein manifolds are \cite{CE,E}, where details and elaborations of the material described here can be found.

To a Weinstein manifold is associated a canonical \emph{skeleton} $\LL$, given as the union of stable manifolds for flow of the Liouville field $Z$. In other words, if we denote by $\phi^t$ the time $t$ flow of $Z$, then the skeleton $\LL_W$ (or just $\LL$ if $W$ is understood) of $W$ is defined by
$$
\xymatrix{
\LL_W=\{x\in W\mid \lim_{t\to\infty}\phi^t(x)\in \on{Crit}(h)\}.
}
$$
The Liouville flow gives a retraction of $W$ onto $\LL.$

Weinstein manifolds are often understood by gluing together Weinstein pairs. A \emph{Weinstein pair} is the data of a Weinstein manifold $W^{2n}$ along with a Weinstein manifold $\Sigma^{2n-2}$ embedded in the ideal contact boundary of $W$, such that the Liouville form on $\Sigma$ is obtained by restriction of the contact form from $\partial W.$ We refer for details to \cite{E} (or to \cite{GPS1}, where these are called \emph{sectors}). There is a notion of skeleton for a Weinstein pair $(W,\Sigma),$ defined by 
$$
\xymatrix{
\LL_{(W,\Sigma)}:=\{x\in W\mid \lim_{t\to \infty}\phi^t(x)\in\on{Crit}(h)\cup \Sigma\}.
}
$$
In other words, the skeleton of a Weinstein pair $(W,\Sigma)$ is the union of $\LL_W$ with the cone (under the Liouville flow) for the skeleton of $\Sigma.$

%

The cosheaf of $\BZ/2$-dg categories $(\mu\Sh^{wr}_\Lambda)_{\BZ/2}$ defined in the previous section is expected to be of use in computing the wrapped Fukaya category $\on{Fuk}^{wr}(W)$ of a Weinstein manifold $W$, defined in the standard way through counts of holomorphic disks. We state this as the following conjecture (an elaboration of the original conjecture of Kontsevich from \cite{SGoHA}):
\begin{conj}\label{conj-cosheaf}
Let $W$ be a Weinstein manifold (or Weinstein pair) with skeleton $\LL$.
\begin{enumerate}
\item There is a cosheaf of $\ZZ/2$-dg categories, which we denote by $\wmsh,$ on the space $\LL$ such that $\wmsh(\LL)$ is equivalent to the wrapped Fukaya category $\on{Fuk}^{wr}(W).$ (If $W$ is a Weinstein pair, this is the partially wrapped category, with stops determined by $\Sigma.$)
\item If $W\cong T^*X$ (with standard cotangent Liouville structure but possibly also with Weinstein pair structure) and we write $\Lambda$ for the skeleton of $T^*X$, then on the space  $\Lambda\cong\LL$, there is an equivalence of cosheaves $\wmsh\cong (\mu\Sh_\Lambda^{wr})_{\BZ/2}.$
\end{enumerate}
\end{conj}
\begin{remarks}\leavevmode
\begin{enumerate}
	\item That the Fukaya category possesses the appropriate covariance properties for inclusions of Weinstein pairs is proved in \cite{GPS1}; a full proof of descent, which would imply part (1) of the conjecture, is expected to appear in a forthcoming sequel to that work.
	\item Since the first appearance of this article, part (2) of this conjecture has been proved in \cite{GPS3}.
\end{enumerate}
\end{remarks}
%
%
%

Part (2) tells us a how to construct the conjectural cosheaf: for each point $p$ in the skeleton $\LL$ of $W$, take some neighborhood $p\in U\subset W$ and an equivalence between $(U,\LL\cap U)\cong (T^*X,\LL_X),$ where $T^*X$ is some cotangent bundle equipped with a Weinstein pair structure with associated skeleton $\LL_X$; then define $\wmsh|_U$ to be the cosheaf $(\mu\Sh^{wr}_{\LL_X})_{\BZ/2},$ and check that the resulting cosheaf is independent of choices. A detailed construction of this cosheaf, through a procedure slightly different to the one described here, can be found in \cite{S}.

In the case of interest to us, the calculation of the cosheaf will be especially easy, since all the singularities which appear in the skeleton we describe for the pair of pants will be arboreal singularities of type $A_m$, for some $m,$ in the sense of \cite{Narb}. The appropriate microlocal sheaf calculation in this case is already known, and the independence of the above construction on choices follows from our earlier discussion of the construction from \cite{Ncyc} of the functor $\calQ.$



\subsection{The permutohedron skeleton}\label{subsec:skel}
In this section, we will show that the quotient of the honeycomb $\frH_{n-1}$ by translations in $\Lambda_n$ actually appears as a skeleton for the $(n-1)$-dimensional pair of pants $\calP_{n-1}$; or equivalently, that the boundary of the tiling of $\RR^n$ by $\frP_{n+1}$ is a skeleton for the universal abelian cover of the pair of pants.

Recall that the standard \emph{$(n-1)$-dimensional pair of pants} is the complex variety
$$
\xymatrix{
\calP_{n-1}=\{z_1+\cdots+z_n+1=0\}\subset(\CC^\times)^n.
}
$$

Define the variety $Y_{n-1}$ by
$$
\xymatrix{
Y_{n-1}:=\{z_1+\cdots+z_n+\frac{1}{z_1\cdots z_n}=0\}\subset(\CC^\times)^n
}.
$$
It has a free action of the group $\ZZ/(n+1)$, generated by $(z_1,\ldots,z_n)\mapsto (\zeta z_1,\ldots,\zeta z_n),$
where $\zeta$ is a primitive $(n+1)^{\text{st}}$ root of unity, whose quotient is the pair of pants.

The reason we begin by studying the $(n+1)$-fold cover $Y_{n-1}$ of $\calP_{n-1}$ instead of the pair of pants itself is that a procedure for constructing a permutohedron skeleton of the former variety has already been described (though not in those terms) in the paper \cite{FU}, so working with $Y_{n-1}$ allows us to appeal to their calculation directly.

The trick from \cite{FU} involves describing the spaces $Y_{n-1}$ inductively: the space $Y_n$ admits a description as the total space of a Lefschetz fibration with fiber $Y_{n-1}.$ As a consequence, we will see that a skeleton for $Y_n$ can be obtained by attaching $n+1$ handles to a skeleton for $Y_{n-1}.$

This Lefschetz fibration is the map
$$
\xymatrix{
Y_n\ar[r]^{p_n}&\CC^\times, & (z_1,\ldots,z_{n+1})\mapsto z_{n+1}.
}
$$
It has $n+2$ critical points $\{(\zeta_k,\ldots,\zeta_k,-(n+1)\zeta_k)\}_{k=0,\ldots,n+1},$ where $\zeta_k$ are $(n+2)^{\text{nd}}$ roots of $\frac{-1}{n+1},$ and hence $n+1$ critical values $\{-(n+1)\zeta_k\}_{k=0,\ldots,n+1}.$
This Lefschetz fibration gives us a very convenient presentation of the Liouville structure on the total space $Y_n$:


\begin{theorem}[\cite{FU}, Theorem 1.5]
	Let $\LL_{Y_n}$ denote the skeleton of $Y_n.$ The restriction to $\LL_{Y_n}$ of the argument projection $\Arg:(\CC^\times)^{n+1}\to T^{n+1}$ to the $(n+1)$-torus is a finite map, and its image $\Arg(\LL_{Y_n})$ divides $T^{n+1}$ into $n+2$ $(n+1)$-permutohedra $\frP_{n+1}$. Moreover, the monodromy of the fibration $p_n$ cylically permutes these permutohedra.
\end{theorem}
\begin{proof}
We will indicate here only the modifications to the argument from \cite{FU} which are necessary in order to understand $\LL_{Y_n}$ as a Weinstein skeleton; the remainder of the calculations can be found there.

The proof is by induction. The base case $n=1$ is clear, so assume the theorem for $Y_{n-1}.$

We can use the Lefschetz fibration $p_n$ to construct a skeleton for $Y_n$ as follows: first, let $U$ be a neighborhood of $S^1\subset \CC^\times_{z_n}$ which does not contain any critical values of $p_n.$ Then $p_n^{-1}(U)$ has a skeleton $\LL'$ which is given by the mapping torus of the monodromy transformation on a skeleton of the general fiber. From our induction hypothesis, we can see that this skeleton divides the $(n+1)$-torus into an oblique cylinder over the $n$-permutohedron.

So far we have described a Liouville structure and skeleton for $p_n^{-1}(U)$; a Liouville structure for the total space $Y_n=p_n^{-1}(\CC^\times)$ comes from extending this Liouville structure over the $n+2$ handles attached at the critical points of $p_n.$ This results in a skeleton $\LL_{Y_n}$ for $Y_n$ obtained by attaching $n+2$ disks to $\LL'.$

The locations of the vanishing cycles along which these disks are glued, and the resulting permutohedra, can be found in \cite{FU}.
\end{proof}
\begin{cor}\label{cor:honeycombskel}
The pair of pants $\calP_n$ has a skeleton $\LL_n$ whose image under $\Arg$ divides the torus $T^{n+1}$ into a single permutohedron; equivalently, the universal abelian cover of $\calP_n$ has a skeleton $\tilLL_n$ whose image under $\Arg$ is the honeycomb lattice $\frH_n.$
\end{cor}
\begin{proof}
The pair of pants $\calP_n$ has a Lefschetz fibration $\bar{p}_n:\calP_n\to\CC^\times/\ZZ/(n+2)\cong \CC^\times$ obtained from the Lefschetz fibration $p_n$ by a $\ZZ/(n+2)$ quotient. (In standard coordinates on $\calP_n,$ this is the map
$(z_1,\ldots,z_{n+1})\mapsto\frac{z_{n+1}^{n+1}}{z_1\cdots z_n}.$)

Hence the skeleton $\LL_n$ can be obtained as the quotient of $\LL_{Y_n}$ by the monodromy transformation, which cyclically exchanges the permutohedra into which $T^{n+1}$ is divided; this gives us the desired description of $\LL_n.$ Moreover, by a diffeomorphism of $T^{n+1}$ (and hence by a symplectomorphism of $T^*T^{n+1}\cong(\CC^\times)^{n+1}$) we can assume this permutohedron is in standard position on $T^{n+1},$ so that $\Arg(\tilLL_n)$ is equal to $\frH_n.$
\end{proof}

Finally, we want to show that our combinatorial cosheaf from Section 2 is the same as the microlocal cosheaf $\wmsh$ described in Section~\ref{sec:A-cosheaf}. This latter cosheaf, for the cover $\tilcalP_{n-1}$ of the pair of pants, is a cosheaf on the space $\LL_{n-1},$ but by pushing forward along $\Arg$ we can equivalently consider this as a cosheaf on $\frH_{n-1}.$

\begin{prop}\label{prop:miccat} There is an equivalence $\wmsh\cong \calQ_{n-1}^{wr}$ of cosheaves of dg categories on the space $\frH_{n-1}.$
\end{prop}
\begin{proof}
Let $p$ be a vertex in $\frH_{n-1}.$ We know that near $p,$ the space $\frH_{n-1}$ (or equivalently, the skeleton $\LL_{n-1}$) is stratified homeomorphic to the $A_n$ arboreal singularity. We need to show that at $p$, the skeleton $\LL_{n-1}$ actually has the correct microlocal sheaf category $(A_n\on{-Perf}_k)_{\BZ/2},$ with the appropriate (co)restriction maps. We can see this from the inductive description of the skeleton $\LL_{n-1}:$ this skeleton was obtained from the mapping torus $M_m$ of a monodromy action on $\LL_{n-2}$ by attaching a disk along a sphere transverse to the singularities of $M_m.$ Hence, by induction we see that there exists a neighborhood $p\in U\subset \calP_{n-1}$ and an equivalence $(U,\LL_{n-1}\cap U)\cong (T^*\RR^{n-1},\calL),$ where $\calL$ is the union of the zero section with the cone on Legendrian lifts of the $n-1$ hyperplanes, taking $p$ to $0$. This establishes the microlocal sheaf calculation, and by $\tilW_{n-1}$ symmetry this is sufficient to prove an equivalence of cosheaves.
\end{proof}

\begin{cor}[``Homological mirror symmetry for the pair of pants'']
There is an equivalence $\MF(\AA^{n+1},W_{n+1})^{\TT^n}\cong\wmsh(\tilLL_{n-1})$ between a category of equivariant matrix factorizations and a category of microlocal sheaves on the universal abelian cover of the pair of pants.
\end{cor}



\bibliographystyle{alpha}
\bibliography{bigbib}
\end{document}